\documentclass[12pt]{amsart}
\usepackage{tikz-cd}
\usepackage{enumitem}
\setlist[description]{leftmargin=\parindent,labelindent=\parindent}
\usetikzlibrary{matrix,arrows,decorations.pathmorphing}
\usepackage{amssymb}
\usepackage{amsmath,amscd}
\usepackage{mathrsfs}
\usepackage{url}
\usepackage{cite}
\usepackage{fullpage}
\usepackage{hyperref}
\usepackage{verbatim}
\usepackage{comment}
\usepackage{fancyvrb}
\usepackage{fvextra}
\usepackage{caption}
\usepackage{colonequals}

\newtheorem{thm}{Theorem}[section]
\newtheorem{prop}[thm]{Proposition}
\newtheorem{lem}[thm]{Lemma}
\newtheorem{cor}[thm]{Corollary}
\newtheorem{conj}[thm]{Conjecture}

\theoremstyle{definition}
\newtheorem{definition}[thm]{Definition}
\newtheorem{example}[thm]{Example}
\newtheorem{rem}[thm]{Remark}
\newtheorem{question}[thm]{Question}

\numberwithin{equation}{section}

\newcommand{\B}{\mathcal{B}}

\newcommand{\zz}{\mathbb{Z}}
\newcommand{\cc}{\mathbb{C}}
\newcommand{\qq}{\mathbb{Q}}

\newcommand{\M}{\mathcal{M}}

\newcommand{\pp}{\mathbb{P}}

\DeclareMathOperator{\cl}{cl}
\newcommand{\et}{\mathrm{\acute{e}t}}

\renewcommand{\O}{\mathcal{O}}

\DeclareMathOperator{\Aut}{Aut}

\DeclareMathOperator{\SL}{SL}

\DeclareMathOperator{\Pic}{Pic}
\DeclareMathOperator{\spec}{Spec}

\DeclareMathOperator{\im}{Im}

\DeclareMathOperator{\tr}{Tr}

\DeclareMathOperator{\Sym}{Sym}

\DeclareMathOperator{\gr}{gr}

\DeclareMathOperator{\BGL}{BGL}

\newcommand{\ff}{\mathbb{F}}
\renewcommand{\hat}{\widehat}

\usepackage{wrapfig}
\renewcommand{\aa}{\mathbb{A}}

\newcommand{\Mb}{\overline{\M}}

\title{Chow rings, cohomology rings, and point counts of moduli spaces of curves}
\author{Hannah Larson}

\begin{document}

\maketitle

\begin{abstract}
In this expository article, we present on state-of-the art results regarding three closely related invariants of moduli spaces of curves: their Chow rings, cohomology rings, and point counts over finite fields. We study the moduli space $\M_{g,n}$, parameterizing smooth genus $g$ curves with $n$ marked points, as well as its compactification by stable curves $\Mb_{g,n}$. After explaining the relationship between these different invariants, we survey what is know regarding the following related questions: When are the Chow rings tautological? When are the cohomology groups tautological? And when are the point counts over fields of size $q$ given by a polynomial in $q$?
\end{abstract}

\section{Introduction}

First conceived of by Riemann in 1857 \cite{Riemann}, the moduli space $\M_g$ of smooth genus $g$ curves is a central object of study in algebraic geometry. Its compactification $\Mb_{g}$, parameterizing stable genus $g$ curves, was introduced by Deligne and Mumford in 1969 \cite{DM}.
The points in the boundary of $\Mb_{g}$ correspond to nodal curves with finite automorphism group. Such curves are built by gluing together smooth curves along marked points. It is therefore natural to study all together the system of moduli spaces $\M_{g,n}$ parameterizing smooth genus $g$ curves with $n$ marked points and their compactifications $\Mb_{g,n}$ parameterizing stable curves of genus $g$ with $n$ marked points.

There are natural maps between these moduli spaces corresponding to gluing
\begin{equation} \label{glue}
\Mb_{g_1,n_1+1} \times \Mb_{g_2,n_2+1} \to \Mb_{g_1+g_2,n_1+n_2} \qquad \text{and} \qquad \Mb_{g-1,n+2} \to \Mb_{g,n}
\end{equation}
and forgetting marked points
\begin{equation} \label{forget}
\Mb_{g,n+1} \to \Mb_{g,n}.
\end{equation}
The boundary of $\Mb_{g,n}$ is stratified according to the topological type of the stable curve, which is encoded in a dual graph $\Gamma$. For example the topological type pictured on the left below corresponds to the dual graph on the right.
\begin{center}
\includegraphics[width=4in]{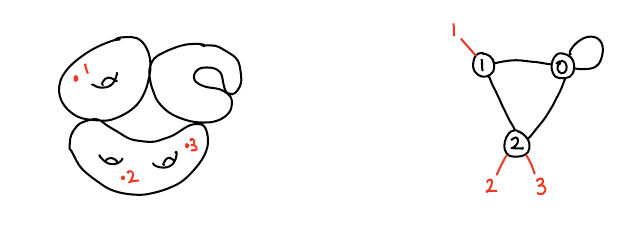}
\end{center}

Given a dual graph $\Gamma$, we define
\begin{equation} \label{gd}
\Mb_{\Gamma} \colonequals \prod_{v \in \Gamma} \Mb_{g_v, n_v}
\end{equation}
where the product ranges over the vertices of $\Gamma$, $g_v$ is the genus of the vertex, and $n_v$ is the number of half-edges incident to the vertex.
The closure of the stratum of curves with topological type $\Gamma$ is the image of the gluing map 
\[\xi_\Gamma \colon \Mb_{\Gamma} \to \Mb_{g,n}\]
given by gluing marked points that are joined by an edge in $\Gamma$. We refer the reader to \cite[Chapter X]{ACG} for further definitions and background on stable curves.

\smallskip
In this article, we study different aspects of the geometry of these moduli spaces of curves, as captured by three closely related invariants.
 The Chow ring, denoted $A^*(\M)$, is an algebraic invariant, capturing information about the algebraic cycles and how they intersect each other. 
The cohomology ring, denoted $H^*(\M)$, is an invariant of the topological space formed by the complex points of the moduli space. Throughout, Chow and cohomology rings will be taken with rational coefficients.
Finally, 
because these moduli spaces are defined over the integers, we may also consider arithmetic invariants. In particular, upon base changing to any finite field $\ff_q$, the moduli space becomes a finite set of points, and we can count the number of such points, denoted $\# \M(\ff_q)$.\footnote{The moduli spaces in question are smooth Deligne--Mumford stacks. For such spaces, one can also define ``stacky" point counts. However, these agree with the usual point counts of their coarse moduli spaces, see Section \ref{sec:glt} for more details.} One can ask to compute each of these invariants for $\M = \M_{g,n}$ or $\M = \Mb_{g,n}$. In Section \ref{sec:rel}, we will explain how these different invariants relate to each other.

The natural maps \eqref{glue} and \eqref{forget} between moduli spaces of curves give rise to a distinguished collection of classes, called tautological classes, which will help to guide our study of their Chow and cohomology rings.

\begin{definition} \label{taut_chow}
The tautological subrings $R^*(\Mb_{g,n}) \subset A^*(\Mb_{g,n})$ are the smallest system of subrings closed under push forward along the gluing and forgetful maps. The tautological subring $R^*(\M_{g,n}) \subset A^*(\M_{g,n})$ is the image of $R^*(\Mb_{g,n})$ under the restriction map $A^*(\Mb_{g,n}) \to A^*(\Mb_{g,n})$.
\end{definition}

\begin{definition}
For either $\M = \Mb_{g,n}$ or $\M_{g,n}$, the
tautological cohomology ring, denoted $RH^*(\M) \subset H^*(\M)$, is the image of $R^*(\M)$ under the cycle class map $A^*(\M) \to H^*(\M)$.
\end{definition}

Compared with the entire Chow or cohomology rings, the tautological subrings are relatively well-understood. They admit additive generators called \emph{decorated boundary strata}. To define them, we first recall the kappa and psi classes. Let $f \colon \Mb_{g,n+1} \to \Mb_{g,n}$ be the map defined by forgetting the last marking and stabilizing. In fact, $f \colon \Mb_{g,n} \to \Mb_{g,n}$ is the universal curve over $\Mb_{g,n}$ and comes 
equipped with $n$ disjoint sections 
\[\sigma_1, \ldots, \sigma_n\colon \Mb_{g,n} \to \Mb_{g,n+1}\]
corresponding to the $n$ markings.
We define
\footnote{To see that the classes $\psi_i$ (and hence also $\kappa_j$) are tautological according to Definition \ref{taut_chow}, we realize the section $\sigma_i$ as a gluing map
$\sigma_i \colon \Mb_{g,n} = \Mb_{g,\{1,\ldots, \hat{i}, \ldots ,n\}} \times \Mb_{0,\{p',i,n+1\}} \to \Mb_{g,n+1}$
which glues $p$ to $p'$. By Definition \ref{taut_chow}, the class $f_*(\sigma_{i*}1)^2$ must be tautological. On the other hand, using the self-intersection formula, we have
$(\sigma_{i*}1)^2 = \sigma_{i*}(c_1(N_{\sigma_i(\Mb_{g,n})/\Mb_{g,n+1})}) = \sigma_{i*}(c_1(\sigma_i^*\omega_f^\vee)) = \sigma_{i*}(-\psi_i)$.
The second equality above follows from identifying the normal bundle to the section with the restriction of the relative tangent bundle to the section. Finally, since $\sigma_i$ is a section of $f$, we have
$R^1(\Mb_{g,n}) \ni f_*(\sigma_{i*}1)^2 = f_*\sigma_{i*}(-\psi_i) = -\psi_i$,
showing that $\psi_i$ is tautological.}
\[\psi_i = c_1(\sigma_i^* \omega_f) \in A^1(\Mb_{g,n}) \qquad \text{and} \qquad \kappa_j = f_*(\psi_{n+1}^{j+1}) \in A^j(\Mb_{g,n}).\]

Given any stable graph $\Gamma$, classes of the following form are tautological by Definition \ref{taut_chow}:
\begin{equation} \label{decst} \xi_{\Gamma*
}\left( \prod_{v \in \Gamma} \mathrm{pr}_v^*(\text{monomial in $\psi_i$ and $\kappa_j$})\right).
\end{equation}
A class of the form \eqref{decst} is called a \emph{decorated boundary stratum}, and can be represented by a decorated stable graph. For example, the decorated stable graph below represents a tautological class in $R^4(\Mb_{6,2})$.
\begin{center}
\includegraphics[width=1.3in]{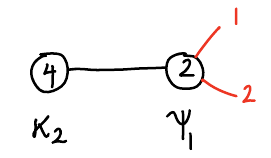}
\end{center}
The name comes from the fact that $\xi_{\Gamma*}(1)$ is the fundamental class of a boundary stratum, but we also allow $\kappa$ and $\psi$ ``decorations". The graph $\Gamma$ is allowed to be the trivial graph with one vertex of genus $g$, which recovers the usual $\psi$ and $\kappa$ classes.
Using excess intersection, an algorithm to express a product of decorated boundary strata as a sum of decorated boundary strata is given in \cite[Appendix A]{GP}. Consequently, one concludes that decorated boundary strata in fact form an additive set of generators for $R^*(\Mb_{g,n})$.
 
There is a systematic set of known relations among decorated boundary strata, known as Pixton's relations, which are conjectured to be all relations \cite{PP, PPZ}. We refer the reader to the excellent survey article \cite{calc} for a description of these tautological relations, their history, and open questions about them.

\medskip
Computing the full Chow ring, cohomology ring, or exact point counts of $\Mb_{g,n}$ or $\M_{g,n}$ are difficult problems which have only been answered in low genera. These low genus results are of interest in their own right, but also play a key role in answering more qualitative questions, in ranges where exact calculation of these invariants seems out of reach. In this article, we will present the current progress on the following related questions.

\begin{question}
For which $g$ and $n$ is the Chow/cohomology ring of $\Mb_{g,n}$ or $\M_{g,n}$ generated by tautological classes?
\end{question}

\noindent
In Section \ref{sec:taut}, we survey a large body of work on this question, which leaves it unanswered in a small, finite number of cases.

\medskip
In our first question above, we fix $g$ and $n$ and ask if the entire Chow/cohomology ring is tautological. A more refined question would be to also fix a codimension in Chow or a cohomological degree.  In Chow, there are relatively few results in this direction. It is known that $ A^1(\Mb_{g,n}) = \Pic(\Mb_{g,n})$ is tautological for all $g$ and $n$. Moreover, if $\Mb_{g,n}$ is rationally connected, then the recent paper \cite[Theorem 1.4]{lu} shows that $A^2(\Mb_{g,n})$ is tautological.
 In Section \ref{sec:coh}, we will present progress on this question when fixing a cohomological degree.

\begin{question}
For which $k, g, n$ is $H^k(\Mb_{g,n})$ generated by tautological classes?
\end{question}

\noindent
Going beyond tautological classes, in Section \ref{sec:coh} we also explain some recent results that describe odd cohomology groups of $\Mb_{g,n}$. These groups are non-tautological, but admit decorated graph generators of a similar flavor.
An understanding of $H^k(\Mb_{g,n})$ can also be used to shed light on the weight $k$ part of the compactly supported cohomology of the interior $\M_{g,n}$, using Deligne's weight spectral sequence. Much progress has been made in this direction as well, see Section \ref{41} for more detailed references.

Finally, in Section \ref{sec:pts}, we connect results on the cohomology of $\Mb_{g,n}$ and $\M_{g,n}$ to their point counts over finite fields. Our motivating qualitative question is the following.

\begin{question}
For which $g, n$ is $\#\Mb_{g,n}(\ff_q)$ or $\#\M_{g,n}(\ff_q)$ a polynomial in $q$?
\end{question}

\noindent
We also explain results of Bergstr\"om--Faber \cite{BF} where exact point counts are known in low genera, including in ranges where they are not polynomial.

\subsection*{Acknowledgements} I am grateful to Samir Canning, Sam Payne, and an anonymous referee for valuable comments on an earlier version of this article. 
This article was written during the period the author served as a Clay Research Fellow.

\section{Relationships between the different invariants} \label{sec:rel}
In this section we collect results comparing the Chow rings, cohomology rings, and point counts of $\Mb_{g,n}$ and $\M_{g,n}$. Specifically, we are interested in the following invariants, and aim to explain connections between them as indicated by edges below:
\begin{center}
\begin{tikzcd}
A^*(\Mb_{g,n}) \arrow[dash]{d} \arrow[dash]{r} & A^*(\M_{g,n})  \arrow[dash]{d} \\
 H^*(\Mb_{g,n}) \arrow[dash]{d} \arrow[dash]{r} & H^*(\M_{g,n}) \text{ or } H^*_c(\M_{g,n}) \arrow[dash]{d} \\
\#\Mb_{g,n}(\ff_q) \arrow[dash]{r} & \#\M_{g,n}(\ff_q).
\end{tikzcd}
\end{center}
Since $\M_{g,n}$ is smooth, $H^*(\M_{g,n})$ and $H^*_c(\M_{g,n})$ are dual to each other.

In nice situations, the invariants in the left column all line up. For example, as we shall see, if $A^*(\Mb_{g,n}) = R^*(\Mb_{g,n})$, then the cycle class map is an isomorphism and $\#\Mb_{g,n}(\ff_q)$ is a polynomial in $q$ whose coefficients are the dimensions of the even cohomology groups. However, a similar statement does not hold for $\M_{g,n}$.
While reading this section, it may help to keep in mind the following simple example.

\begin{example} \label{proj}
Projective space $\pp^n$ is smooth and proper over $\spec \zz$.
The Chow and cohomology rings of $\pp^n$ are both generated by the first Chern class of the tautological line bundle, denoted $\zeta = c_1(\O_{\pp^n}(1))$. In fact, we have
\[A^*(\pp^n) \cong H^*(\pp^n) = \qq[\zeta]/(\zeta^{n+1}).\]
By writing $\pp^n = \aa^n \sqcup \aa^{n-1} \sqcup \cdots \sqcup \aa^1 \sqcup \aa^0$,
one also sees that
\[\#\pp^n(\ff_q) = q^n + q^{n-1} + \ldots + q + 1,\]
which is a polynomial in $q$. Notice that the odd cohomology of $H^*(\pp^n)$ vanishes and the ranks of the even degree cohomology groups are the coefficients appearing in the polynomial above. This is an example of how the invariants line up nicely for a smooth, proper space.

Now consider $X = \pp^2 \smallsetminus E$ for $E$ an elliptic curve. By excision, $A^*(X)$ is still generated by $\zeta$. In fact, since we work with rational coefficients, $A^*(X) = \qq$. However, the Chow and cohomology rings are no longer isomorphic. Moreover,
\[\#X(\ff_q) = \#\pp^2(\ff_q) - \#E(\ff_q)\]
is not a polynomial, as $ \#E(\ff_q)$ is not a polynomial. Although the Chow ring of $X$ is still relatively simple, its cohomology and point counts are more sensitive to the removal of $E$. In general, we will see that comparing the different invariants is less straightforward for spaces that are not proper.

\end{example} 
\subsection{Comparing Chow of $\Mb_{g,n}$ and $\M_{g,n}$}
The excision sequence for Chow groups says there is a right-exact sequence
\[A^{*-1}(\partial \M_{g,n}) \to A^*(\Mb_{g,n}) \to A^*(\M_{g,n}) \rightarrow 0,\]
where $\partial \M_{g,n} \colonequals \Mb_{g,n} \smallsetminus \M_{g,n}$ denotes the boundary.
Thus, by Definition \ref{taut_chow}, we immediately see the following.
\begin{lem} \label{easy}
If $A^*(\Mb_{g,n}) = R^*(\Mb_{g,n})$ then $A^*(\M_{g,n}) = R^*(\M_{g,n})$.
\end{lem}
The converse need not hold. For a concrete example, we will see in Section \ref{sec:taut} that
$A^*(\M_{2,10}) = R^*(\M_{2,10})$ but $A^*(\Mb_{2,10}) \neq R^*(\Mb_{2,10})$.

\smallskip
Nevertheless, the excision sequence shows that if
$A^*(\M_{g,n}) = R^*(\M_{g,n})$ and all classes in the image of $A^{*-1}(\partial \M_{g,n}) \to A^*(\Mb_{g,n}) $ are tautological, then
$A^*(\Mb_{g,n}) = R^*(\Mb_{g,n})$. 
Let $\widetilde{\partial \M_{g,n}}$ denote the normalization of the boundary.
Since push forward along proper surjective maps is surjective on Chow groups, the
image of  $A^{*-1}(\partial \M_{g,n}) \to A^*(\Mb_{g,n})$ is the same as the image of $A^{*-1}(\widetilde{\partial \M_{g,n}}) \to A^*(\Mb_{g,n})$.
Moreover, $\widetilde{\partial \M_{g,n}}$ is the disjoint union of all possible gluing maps
\[\Mb_{g_1,n_1+1} \times \Mb_{g_2,n_2+1} \to \Mb_{g_1+g_2,n_1+n_2} \qquad \text{and} \qquad \Mb_{g-1,n+2} \to \Mb_{g,n}.\]
At this point, to compute the Chow ring of $\widetilde{\partial \M_{g,n}}$,
we need to understand the Chow ring of a product of two moduli spaces. In general, the K\"unneth formula does not hold for Chow rings; instead, we make the following definition.
\begin{definition}
We say $X$ has the \emph{Chow--K\"unneth generation Property (CKgP)} if for all $Y$, the tensor product map 
\[A^*(X) \otimes A^*(Y) \to A^*(X \times Y)\]
is surjective.
\end{definition}
The CKgP was introduced in work of Totaro \cite{Totaro} and Bae--Schmitt \cite{BS}, and plays a key role in the author's work with Canning \cite{ckgp} on the Chow rings of moduli spaces of curves. The following result
illustrates what additional input may yield a converse to Lemma \ref{easy}.

\begin{lem}[Lemma 4.1 of \cite{ckgp}] \label{assemble}
Let $g$ and $n$ be given. Suppose $A^*(\M_{g',n'}) = R^*(\M_{g',n'})$ and $\M_{g',n'}$ has the CKgP for all $g',n'$ with $g' \leq g$ and $2g' + n' \leq 2g + n$. Then $A^*(\Mb_{g,n}) = R^*(\Mb_{g,n})$ and $\Mb_{g,n}$ has the CKgP.
\end{lem}
\begin{proof}
The boundary of $\Mb_{g,n}$ is surjected onto by the disjoint union of $\Mb_{g_1,n_1+1} \times \Mb_{g_2,n_2+1}$ and $ \Mb_{g-1,n+2}$. Each of the moduli spaces involved has either genus or number of markings smaller than $g$ or $n$. Thus, we can inductively assume that they all have the CKgP and tautological Chow rings. The CKgP implies that the Chow ring of $\Mb_{g_1,n_1+1} \times \Mb_{g_2,n_2+1}$ is generated by pullbacks of tautological classes from the two factors. Since the tautological rings are closed under pushforward along gluing maps, it follows that all classes supported on the boundary of $\Mb_{g,n}$ are tautological, which implies $A^*(\Mb_{g,n}) = R^*(\Mb_{g,n})$. To see that $\Mb_{g,n}$ also has the CKgP, we use the fact that the CKgP plays well with stratifications and surjective maps, see \cite[Lemma 4.1]{ckgp} for more details.
\end{proof}

\begin{rem}
When $g$ and $n$ are small, the moduli spaces $\M_{g,n}$ are unirational, and many of them can be studied very explicitly.
More precisely, in the cases where it has been shown that $A^*(\M_{g,n}) = R^*(\M_{g,n})$, it is because the moduli space can be stratified into pieces, each of which is surjected onto by a space whose Chow ring is not hard to calculate, such as an open substack of a Grassmann bundle over $\BGL_r$. Such tractable spaces tend to have the CKgP. Moreover, the CKgP plays well with stratifications and surjective maps, so an argument that  $A^*(\M_{g,n}) = R^*(\M_{g,n})$ tends to also show
that $\M_{g,n}$ has the CKgP, cf Remark \ref{rem:ckgp}. When $g$ and $n$ are large however, the moduli spaces $\M_{g,n}$ are of general type, so in particular do not admit such nice unirational parameterizations. While this means the techniques used for small $g$ and $n$ necessarily break down for large $g$ and $n$, we do not know of direct arguments in general for the existence of non-tautological classes or failure of CKgP just based on the birational geometry of the moduli space.
\end{rem}

\subsection{Comparing Chow and cohomology} \label{sec:cl}
The cycle class map $\cl: A^*(\M) \to H^*(\M)$ is a ring homomorphism whose image lies in even degree. In general, the cycle class map need not be injective or surjective, but we have the following lemmas.

\begin{lem}[Lemma 3.11 of \cite{ckgp}]
If $X$ is a smooth, proper Deligne--Mumford stack over $\cc$ and $X$ has the CKgP, then $\cl$ is an isomorphism.
\end{lem}
\begin{proof}
Since $X$ has the CKgP, the tensor product map $A^*(X) \otimes A^*(X) \to A^*(X \times X)$ is surjective. In particular,
we obtain an algebraic decomposition of the class of the diagonal in $X \times X$, which implies that the cycle class map is an isomorphism.
\end{proof}

In Example \ref{proj}, $\pp^n$ has the CKgP.

\smallskip
Somewhat surprisingly, another sufficient criterion for the cycle class map to be an isomorphism is that the Chow ring is a countable $\qq$ vector space. This is essentially a consequence of work of Kimura and Totaro. The precise references and form stated below can be found in \cite[Lemma 7.1]{h11}.
\begin{lem}[Lemma 7.1 of \cite{h11}] \label{cl_iso}
If $X$ is a smooth, proper Deligne--Mumford stack over $\cc$ and $A^*(X)$ is a countable $\qq$ vector space, then $\cl$ is an isomorphism.
\end{lem}

Because decorated boundary strata generate $R^*(\Mb_{g,n})$, the ring is finitely generated. Thus, we have the following corollary.
\begin{cor} \label{rhcor}
If $A^*(\Mb_{g,n}) = R^*(\Mb_{g,n})$, then $H^*(\Mb_{g,n}) = RH^*(\Mb_{g,n})$.
\end{cor}

We note that the analogue of Corollary \ref{rhcor} is not true for $\M_{g,n}$. For example, we have that $A^*(\M_{0,4}) = R^*(\M_{0,4}) = \qq$ but $H^1(\M_{0,4}) \neq 0$, so $H^*(\M_{0,4}) \neq RH^*(\M_{0,4})$.

\begin{rem}
For $\M = \Mb_{g,n}$ or $\M_{g,n}$, it is conjectured that Pixton's relations are complete in Chow in \cite[Conjecture 2]{Pixton}. It has also been conjectured that Pixton's relations are complete in cohomology. This conjecture would imply in particular that 
the cycle class map restricts to an isomorphism on $R^*(\M) \to RH^*(\M)$. While this is known to hold in some cases, such as in the situation of Corollary \ref{rhcor}, it remains open in general.
\end{rem}

For the open moduli spaces $\M_{g,n}$, the cycle class map is rarely an isomorphism. We write $W_k H^i(X)$ for the weight $k$ part of Deligne's weight filtration. The image of the cycle class map is always contained in the lowest weight subspace $W_kH^k(X)$. 
We say that the cohomology is \emph{Tate type} if its semisimplification is a sum of Tate twists (see Section \ref{sec:glt}).
The following lemma can be thought of as a weaker replacement of Lemma \ref{cl_iso} for spaces which are not necessarily proper.
\begin{lem}[Lemma 4.3 of \cite{ste}] \label{open-ckgp}
Let $X$ be an open substack of a smooth, proper Deligne--Mumford stack over $\cc$.
If $X$ has the CKgP, then
\[\cl: A^*(X) \to \bigoplus_{k} W_k H^k(X)\]
is surjective. In particular, if $k$ is odd, then $W_k H^k(X) = 0$ and if $k$ is even, then $W_k H^k(X)$ is Tate type.
\end{lem}

\subsection{Comparing cohomology of $\Mb_{g,n}$ and $\M_{g,n}$} \label{wt}
In contrast with Lemma \ref{easy}, the restriction map $H^*(\Mb_{g,n}) \to H^*(\M_{g,n})$ is usually not surjective. For example, $H^1(\Mb_{0,4}) = 0$ but $H^1(\M_{0,4}) \neq 0$. The key to comparing the cohomology of $\Mb_{g,n}$ and $\M_{g,n}$ is Deligne's weight spectral sequence. In this section, we explain how the graded pieces of the weight filtration are determined by the cohomology of the $\Mb_{g',n'}$ with $2g' + n' \leq 2g + n$, along with their $\mathbb{S}_{n'}$ actions and pullback maps. For the later applications, it will be convenient to state results for compactly supported cohomology $H^*_c(\M_{g,n})$ and its weight filtration, but since $\M_{g,n}$ is smooth, this is dual to ordinary cohomology.

We write 
\[\gr_k H^*_c(\M_{g,n}) \colonequals W_k H^*_c(\M_{g,n})/W_{k-1}H_c^*(\M_{g,n})\]
for the weight $k$ graded piece. We now explain how the weight $k$ graded piece of cohomology is determined by the pullback maps on the $k$th cohomology of all boundary strata of $\Mb_{g,n}$, following \cite[Section 2.3]{PW}.
This is a special case of a more general method for computing compactly supported cohomology groups of an open stratum in stratified spaces, explained in \cite{Pstrat}, which also makes precise the idea of taking the Poincar\'e dual of Deligne's weight spectral sequence.

Let $\Gamma$ be a dual graph and let $\Mb_{\Gamma}$ be the product of moduli spaces as in \eqref{gd}. Then $\Aut \Gamma$ acts on $H^k(\Mb_{\Gamma})$ by permuting the point in a way corresponding to the action of $\Aut \Gamma$ on half-edges. We also require the representation $\det E(\Gamma)$, where each element of $\Aut \Gamma$ acts by the sign of the corresponding permutation of the edges of $\Gamma$. The Poincar\'e dual of Deligne's weight spectral sequence gives rise to a complex
\[H^k(\Mb_{g,n}) \xrightarrow{d_0} \bigoplus_{|E(\Gamma)| = 1} H^k(\Mb_{\Gamma} \otimes \det E(\Gamma))^{\Aut \Gamma} \xrightarrow{d_1} 
\bigoplus_{|E(\Gamma)| = 2} H^k(\Mb_{\Gamma} \otimes \det E(\Gamma))^{\Aut \Gamma}  \xrightarrow{d_2}  \dots.
\]
Here, $d_0$ is the pullback to the normalization of the boundary. 
When there is an edge contraction $\Gamma \to \Gamma'$, the factor of $d_1$ between those corresponding components is given by the associated pullback, with an appropriate sign.
The $j$th homology of this complex computes
\[\gr_k H_c^{j+k}(\M_{g,n}) = \frac{\ker d_j}{\im d_{j-1}}.\]

\subsection{Comparing cohomology and point counts} \label{sec:glt}
Given a Deligne--Mumford stack $X$ of finite type over $\zz$, one defines its number of points over $\ff_q$ as
\[\#X(\ff_q) \colonequals \sum_{\xi \in [X(\ff_q)]} \frac{1}{|\Aut(\xi)|},\]
where the sum ranges over representatives of isomorphism classes of objects in $X(\ff_q)$. This ``stacky" point count agrees with the regular point count of the coarse moduli space of $X$ \cite{Beh}.

The essential tool relating point counts and cohomology is the Grothendieck--Lefschetz trace formula. 
The inspiration and name comes from the classical Lefschetz fixed point theorem in topology, which computes the size of the fixed locus of an endomorphism in terms of the trace of its pullback acting on cohomology. If $\phi: X \to X$ is an endomorphism, then the fixed locus of $\phi$ is the intersection of the diagonal $\Delta \subset X \times X$ with the graph $\Gamma_\phi \subset X \times X$. We say that the fixed locus is \emph{simple} if the intersection of $\Delta$ and $\Gamma_\phi$ is transverse, equivalently if the fixed locus consists of reduced, isolated fixed points.

\begin{thm}[Lefschetz fixed point theorem]
Let $X$ be a compact, oriented manifold. If $\phi: X \to X$ is an endomorphism with simple fixed locus, then
\[\#\text{(fixed points of $\phi$)} = \sum_i (-1)^i \tr(\phi^* | H^i(X)).\]
\end{thm}

To explain the connection with point counts over finite fields, recall that given a variety $X$ over $\mathbb{F}_q$, the $\ff_q$-points of $X$ are precisely the fixed locus of the Frobenius endomorphism. Although $X$ is not longer a compact, oriented manifold, a similar formula for the size of the fixed locus holds in an appropriate cohomology theory. We write $H^i_{c, \et}(-, \qq_\ell)$ for compactly supported $\ell$-adic cohomology.  We refer the reader to \cite{milne} for background on \'etale cohomology. 

\begin{thm}[Behrend--Grothendiek--Lefschetz trace formula \cite{Beh}]
Let $X$ be a smooth Deligne--Mumford stack over $\spec \zz$ and let $\ell$ be a prime with $(\ell, q) = 1$. Then
\begin{equation} \label{trace} \#X(\ff_q) = \sum_i (-1)^i \tr(\mathrm{Frob}_q^* | H^i_{c, \et}(X_{\overline{\ff}_q}, \qq_\ell)).
\end{equation}
\end{thm}

The compactly supported \'etale cohomology groups also admit a weight filtration.
Work of Deligne \cite{deligne} states that the eigenvalues of $\mathrm{Frob}_q^*$ acting on the weight $k$ graded piece are all Weil numbers of of weight $k$, meaning that under any embedding $\overline{\qq_\ell} \hookrightarrow \cc$,
they are algebraic integers of absolute value $q^{k/2}$. Moreover, the characteristic polynomial of $\mathrm{Frob}_q^*$ is independent of the choice of the auxiliary prime $\ell$ that goes into the construction of $\ell$-adic \'etale cohomology.

\begin{example}
Let us count $\ff_q$-points on $\pp^1$ using the trace formula. The endomorphism $\mathrm{Frob}_q: \pp^1 \to \pp^1$ is given by $[x:y] \mapsto [x^q:y^q]$.
We have $H^0_{c,\et}(\pp^1_{\overline{\ff}_q}, \qq_\ell)$ is spanned by the fundamental class of $\pp^1$, which pulls back to the fundamental class of $\pp^1$ under $\mathrm{Frob}_q$. Meanwhile, $H^1_{c,\et}(\pp^1_{\overline{\ff}_q}, \qq_\ell) = 0$, and $H^2_{c,\et}(\pp^1_{\overline{\ff}_q}, \qq_\ell)$ is spanned by the class of a point.
Since $\mathrm{Frob}_q$ is degree $q$, the pullback of a point is $q$ times the class of a point. Thus
\[\#\pp^1(\ff_q) = \sum_i (-1)^i \tr(\mathrm{Frob}_q^* | H^i_{c, \et}(\pp^1_{\overline{\ff}_q}, \qq_\ell)) = 1 - 0 + q = 1 + q.\]
\end{example}

In what follows, we write $\mathsf{L} \colonequals H^2_{c,\et}(\pp^1, \qq_\ell)$
 and $\mathsf{L}^i \colonequals \mathsf{L}^{\otimes i}$, called the $i$th Tate twist (often also denoted $\qq_\ell(-i)$). The $i$th Tate twist is a $1$-dimensional Galois representation on which $\mathrm{Frob}_q^*$ acts by $q^i$.
If the semi-simplification of $H^*_{c,\et}(X, \qq_\ell)$ is a sum of Tate twists, we say the cohomology of $X$ is \emph{Tate type}.
Generalizing the above example, the following lemma is an immediate consequence of the trace formula.

\begin{lem} \label{tate}
Let $X$ be a smooth Deligne--Mumford stack over $\spec \zz$.
If the cohomology of $X$ is Tate type, then $\#X(\ff_q)$ is a polynomial in $q$.
\end{lem}

As we saw in the example with the class of a point on $\pp^1$, 
algebraic cycles on $X$ defined over $\zz$ always give classes of Tate type. In particular, we have the following corollary.
\begin{cor}[Proposition 3.15 of \cite{ckgp}] \label{hrh}
If $H^*(\Mb_{g,n}) = RH^*(\Mb_{g,n})$ then $\#\Mb_{g,n}(\ff_q)$ is a polynomial in $q$.
\end{cor}

Less obvious is that the converse of Lemma \ref{tate} holds when $X$ is also assumed to be proper. In fact, having polynomial point count determines the cohomology of $X$ as a Galois representation.
\begin{thm}[van den Bogaart--Edixhoven \cite{edix}] \label{be}
Let $X$ be a smooth, proper Deligne--Mumford stack over $\spec \zz$.
If $\#X(\ff_q) = a_d q^d + \cdots + a_1 q + a_0$ is a polynomial in $q$ , then all odd cohomology groups of $X$ vanish and the even cohomology groups are pure Tate
\[H^{2i}_{\et}(X, \qq_\ell) = (\mathsf{L}^i)^{\oplus a_{i}}.\]
\end{thm}

\begin{rem}
As we shall see later on, in every case where $\#\Mb_{g,n}(\ff_q)$ is known to be a polynomial in $q$, it is also known that $H^*(\Mb_{g,n}) = RH^*(\Mb_{g,n})$. Theorem \ref{be} says that  if $\#\Mb_{g,n}(\ff_q)$ is a polynomial in $q$, then the cohomology of $X$ is pure Tate. The Tate conjecture then predicts that the cycle class map is surjective. However, the statement that $H^*(\Mb_{g,n}) = RH^*(\Mb_{g,n})$ is stronger than the claim that the cycle class map is surjective.
\end{rem}

One of the key ways properness is used in Theorem \ref{be} is that, when $X$ is proper, weight $k$ eigenvalues only occur in $H^k$. When $X$ is not proper, the converse to Lemma \ref{tate} need not hold because, a priori, some non-Tate eigenvalues could occur equally often in odd and even cohomological degrees, and cancel each other in the alternating sum \eqref{trace}.
Motivated by the possibility of such cancellation, we define the \emph{weight $k$ Euler characteristic of $X$} to be
\begin{equation} \label{chik-def} \chi_k(X) = \sum_i (-1)^i \dim \gr_k H^i_c(X).
\end{equation}
With this notation, we have the following result, even if $X$ is not proper. (If $X$ is proper, then the result follows immediately from the contrapositive of Theorem \ref{be}.)
\begin{lem}[Proposition 2.3 of \cite{h13}] \label{chik}
Let $X$ be a smooth Deligne--Mumford stack over $\spec \zz$. If $\#X(\ff_q)$ is a polynomial in $q$, then $\chi_k(X) = 0$ for all odd $k$.
\end{lem}

\begin{example}
The first example of an interesting odd cohomology group of moduli spaces of curves is 
\[H^{11}(\Mb_{1,11}) \otimes \cc = H^{11,0}(\Mb_{1,11}) \oplus H^{0,11}(\Mb_{1,11}).\]
Let
\[\Delta(q) = q\prod_{m=1}^\infty (1 - q^m)^{24} = \sum_{n=1}^\infty \tau(n) q^n\]
be the weight $12$ cusp form for $\SL_2\zz$. (We apologize for the  clash of notation; in this one instance above, $q = e^{2\pi i z}$ for a complex variable $z$.)
From the weight $12$ cusp form, one can construct an explicit holomorphic form on $\Mb_{1,11}$, see \cite[p. 306]{FP}.
Meanwhile, having odd weight, the Galois representation $H_{\et}^{11}(\Mb_{1,11},\qq_\ell)$ is not Tate type. In fact, it is a $2$-dimensional irreducible Galois representation which satisfies
\[\tr(\mathrm{Frob}_p^* | H^{11}_{\et}(\Mb_{1,11}, \qq_\ell)) = \tau(p).\]
It turns out that all other odd cohomology groups of $\Mb_{1,11}$ vanish and the even cohomology groups are Tate type.
Consequently, the point counts of $\Mb_{1,11}$ have the form
\[\#\Mb_{1,11}(\ff_p) = P(p) - \tau(p)\]
where $P(p)$ is a polynomial in $p$, see \cite[Section 2.2]{FP} for more details.
\end{example}

More generally, for each Hecke eigenform $f$ of weight $k$ for $\SL_2\zz$, there is an associated $2$-dimensional irreducible Galois representation of weight $k-1$ on which $\mathrm{Frob}_p$ acts by the $p$th Fourier coefficient of $f$ \cite{Del71}. Following the notation of \cite{BF}, for the eigenforms of weights $k = 12, 16, 18$ and $20$, we denote this associated Galois representation by $\mathsf{S}_k$.
There are also irreducible Galois representations, denoted $\mathsf{S}_{j,k}$ of weight $j + 2k - 3$, which come from genus $2$ and
are associated to Siegel cusp forms for $\mathrm{Sp}_4(\zz)$ \cite{Lau05,Wei05}.

This association is part of a larger web of conjectures in the Langlands program which says roughly that there should be a bijection between algebraic automorphic representations and irreducible geometric Galois representations. A result of 
Chenevier and Lannes classifies all cuspidal automorphic representations of conductor $1$ and weight up to $22$. In \cite{BF}, Bergstr\"om and Faber explain how, in combination with this classification result, these Langlands conjectures would then restrict the Galois representations of weight at most $22$ that appear in the cohomology of Deligne--Mumford stacks that are smooth and proper over $\spec \zz$.
 
 \begin{conj} \label{lc}
 Let $X$ be a smooth, proper Delgne--Mumford stack over $\spec \zz$, such as $\Mb_{g,n}$.
 Suppose $k \leq 22$. Then, the semi-simplification
 $H^k(X, \qq_\ell)^{\mathrm{ss}}$ is a sum of Tate twists of the following Galois representations
 \[ \mathbf{1}, \ \mathsf{S}_{12}, \  \mathsf{S}_{16}, \  \mathsf{S}_{18}, \  \mathsf{S}_{20}, \mathsf{S}_{6,8}, \ \mathsf{S}_{22}, \ \mathsf{S}_{4,10},\  \mathsf{S}_{8,8}, \ \mathsf{S}_{12,6}, \ \Sym^2 \mathsf{S}_{12}.\]
\end{conj}

For $X$ a smooth and proper Delgne--Mumford stack of dimension $d$, we have $H^k(X, \qq_\ell)$ is pure weight $k$, meaning the eigenvalues of $\mathrm{Frob}_q$ acting on $H^k(X, \qq_\ell)$ all have absolute value $q^{k/2}$. Thus the leading order terms in the point counts $\#X(\ff_q)$ come from the highest degree cohomology groups.
Poincar\'e duality shows that, as Galois representations, the high degree cohomology groups differ from corresponding low degree cohomology groups by a Tate twist: $H^{2d - k}(X, \qq_\ell) \cong \mathsf{L}^{d - k} \otimes H^k(X, \qq_\ell)$.
Thus, using the trace formula, Conjecture \ref{lc} implies loosely speaking that $\#X(\ff_q)$ is well-approximated by a polynomial in $q$, with the leading order correction terms coming from coefficients of moduli forms and Siegel modular forms.

\subsection{Comparing point counts of $\Mb_{g,n}$ and $\M_{g,n}$} \label{equiv}
The point counts of $\Mb_{g,n}$ and $\M_{g,n}$ are related by the equation
\[\#\Mb_{g,n}(\ff_q) = \#\M_{g,n}(\ff_q) + \#\partial \M_{g,n}(\ff_q).\]
Furthermore, $\#\partial \M_{g,n}(\ff_q)$ can be computed by summing up point counts for each locally closed boundary stratum. Each boundary stratum is a finite group quotient of a product of $\M_{g',n'}$ with $2g' + n' \leq 2g + n$. The point count of a product is the product of the point counts, so one might hope to translate between the following two sets of information
\[\{\#\M_{g',n'}(\ff_q) : 2g'+n' \leq c\} \overset{?}{\longleftrightarrow} \{\#\Mb_{g',n'}(\ff_q) : 2g'+n' \leq c\}.\]
However, some additional care is needed because point counts do not always play well with quotients by finite groups. Instead, we shall see that the sets of information above become equivalent when we pass to their equivariant versions.

The reason we need equivariant point counts is because, although it holds for  \emph{connected} linear algebraic groups, the naive equality $\#[X/G](\ff_q) = \#X(\ff_q)/\#G(\ff_q)$ need not hold for finite groups $G$. For example, consider $G = \zz/2$ acting on $\aa^1$, where the non-trivial element acts by $t \mapsto -t$. The quotient stack $[\aa^1/G]$ has coarse moduli space $\aa^1$. Thus, we find
\[\#\aa^1(\ff_q) = q \qquad \text{and} \qquad \#[\aa^1/G](\ff_q) = \#\aa^1(\ff_q) = q.\]
The resolution to this seemly strange statement is that an $\ff_q$-point of the quotient could come either from two $\ff_q$ points on $\aa^1$ that form a $G$ orbit, or from two conjugate $\ff_{q^2}$ points whose orbit is still defined over $\ff_q$.

In general, equivariant point counts are defined using \emph{twisted forms}.
Given a variety $X$ over $\ff_q$ and $\sigma \in \Aut(X)$, there is a unique twisted form of $X$, denoted $X^\sigma$ with an isomorphism $X^\sigma_{\overline{\ff}_q} \to X_{\overline{\ff}_q}$ that identifies the geometric Frobenius action on $X^\sigma_{\overline{\ff}_q}$ with the composition of $\sigma$ and the Frobenius on $X_{\overline{\ff}_q}$. 
For $X = \Mb_{g,n}$ or $\M_{g,n}$ and $\sigma \in G = \mathbb{S}_n$ and some fixed $q$,
there is a twisted form $X^\sigma$, which is a smooth Deligne--Mumford stack over $\ff_q$ representing the functor taking a scheme $S$ over $\ff_q$ to the set of isomorphism classes of families of (stable) curves over $S$ with $n$ disjoint labeled geometric sections such that Frobenius acts on the sections via the permutation $\sigma$.
 Given a finite group $G$ acting on $X$, conjugate elements of $G$ induce isomorphic twisted forms of $X$, so $\#X^\sigma(\ff_q)$ is a class function on $G$. The \emph{$G$-equivariant point count}, denoted $\#^GX(\ff_q)$, is the element of the representation ring of $G$ associated to the class function $\sigma \mapsto \#X^\sigma(\ff_q)$.
By the Grothendeick--Lefschetz trace formula, $\#X(\ff_q)$ is determined by the virtual Galois representation $\chi_c(X, \qq_\ell) \colonequals \sum (-1)^i H_c^i(X, \qq_\ell)$.  Similarly, knowing the $G$-equivariant point counts amount to knowing $\chi_c(X, \qq_\ell)$ as a $G$-equivariant Galois representation.

The Getzler--Kapranov formula \cite{gk} gives a recipe for computing one of the sets of equivariant point counts below from the other:
\[\{\#^{\mathbb{S}_{n'}}\M_{g',n'}(\ff_q) : 2g'+n' \leq c\} \longleftrightarrow \{\#^{\mathbb{S}_{n'}}\Mb_{g',n'}(\ff_q) : 2g'+n' \leq c\}.\]
The Getzler--Kapranov formula is essentially a combinatorial version of the spectral sequence discussed in Section \ref{wt}.

\section{When are the Chow or cohomology rings tautological?} \label{sec:taut}

In this section, we survey some results on the Chow and cohomology rings of $\Mb_{g,n}$ and $\M_{g,n}$, with particular attention to when the ring is generated by tautological classes or not. We start with the case of the Chow ring, summarizing the results with the chart below.
We then note where some stronger results are known in cohomology.
In the chart below, we summarize the known results by placing a symbol from the key at position $(g, n)$ when the corresponding condition is known to hold for $\Mb_{g,n}$ or $\M_{g,n}$.

\begin{figure}
\begin{center}
\begin{tikzpicture}[scale=.8]
\filldraw (0, 0) circle (4pt);
\node at (4,0) {$A^*(\Mb_{g,n}) = R^*(\Mb_{g,n})$};
\draw[ultra thick] (0, -1) circle (4pt);
\node at (4,-1) {$A^*(\M_{g,n}) = R^*(\M_{g,n})$};
\filldraw (10, -.5) circle (4pt);
\node at (12,-.5) {$\Longrightarrow$};
\draw[ultra thick] (14, -.5) circle (4pt);

\node[scale = 1.3, color = red] at (10, -2.5) {$\boldsymbol{\times}$};
\node at (12,-2.5) {$\Longrightarrow$};
\node[scale = .75, color = red] at (14, -2.5) {$\boldsymbol{\times}$};

\node[scale = 1.3, color = red] at (0, -2) {$\boldsymbol{\times}$};
\node at (4,-2) {$A^*(\M_{g,n}) \neq R^*(\M_{g,n})$};
\node[scale = .75, color = red] at (0, -3) {$\boldsymbol{\times}$};
\node at (4,-3) {$A^*(\Mb_{g,n}) \neq R^*(\Mb_{g,n})$};
\end{tikzpicture}
\begin{tikzpicture}[scale = .8]

\filldraw[color=white] (3, 11) circle (4pt);
\filldraw[color=white] (4, 9) circle (4pt);
\filldraw[color=white] (5, 7) circle (4pt);
\filldraw[color=white] (6, 5) circle (4pt);

\filldraw[color=white] (7, 1) circle (4pt);
\filldraw[color=white] (7, 2) circle (4pt);
\filldraw[color=white] (7, 3) circle (4pt);
\draw[color=black,ultra thick] (7, 1) circle (4pt);
\draw[color=black,ultra thick] (7, 2) circle (4pt);
\draw[color=black,ultra thick] (7, 3) circle (4pt);

\node[scale = .85] at (1, -.7) {$1$};
\node[scale = .85] at (2, -.7) {$2$};
\node[scale = .85] at (3, -.7) {$3$};
\node[scale = .85] at (4, -.7) {$4$};
\node[scale = .85] at (5, -.7) {$5$};
\node[scale = .85] at (6, -.7) {$6$};
\node[scale = .85] at (7, -.7) {$7$};
\node[scale = .85] at (8, -.7) {$8$};
\node[scale = .85] at (9, -.7) {$9$};
\node[scale = .85] at (10, -.7) {$10$};
\node[scale = .85] at (11, -.7) {$11$};
\node[scale = .85] at (12, -.7) {$12$};
\node[scale = .85] at (13, -.7) {$13$};
\node[scale = .85] at (14, -.7) {$14$};
\node[scale = .85] at (15, -.7) {$15$};
\node[scale = .85] at (16, -.7) {$16$};
\node[scale = .85] at (17, -.7) {$17$};

\node[scale = .85] at (-.7, 1) {$1$};
\node[scale = .85] at (-.7, 2) {$2$};
\node[scale = .85] at (-.7, 3) {$3$};
\node[scale = .85] at (-.7, 4) {$4$};
\node[scale = .85] at (-.7, 5) {$5$};
\node[scale = .85] at (-.7, 6) {$6$};
\node[scale = .85] at (-.7, 7) {$7$};
\node[scale = .85] at (-.7, 8) {$8$};
\node[scale = .85] at (-.7, 9) {$9$};
\node[scale = .85] at (-.7, 10) {$10$};
\node[scale = .85] at (-.7, 11) {$11$};
\node[scale = .85] at (-.7, 12) {$12$};
\node[scale = .85] at (-.7, 13) {$13$};
\node[scale = .85] at (-.7, 14) {$14$};
\node[scale = .85] at (-.7, 15) {$15$};
\node[scale = .85] at (-.7, 16) {$16$};
\node[scale = .85] at (-.7, 17) {$17$};
\node[scale = .85] at (-.7, 18) {$18$};
\node[scale = .85] at (-.7, 19) {$19$};
\node[scale = .85] at (-.7, 20) {$20$};

\draw[->] (0, 0) -- (18, 0);
\draw[->] (0, 0) -- (0, 21);
\node[scale=1.2] at (18.4,0) {$g$};
\node[scale=1.2] at (0, 21.4) {$n$};
\draw (-.1, 1) -- (.1, 1);
\draw (-.1, 2) -- (.1, 2);
\draw (1, -.1) -- (1, .1);
\draw (2, -.1) -- (2, .1);
\draw (3, -.1) -- (3, .1);
\filldraw (0, 3) circle (4pt);
\filldraw (0, 4) circle (4pt);
\filldraw (0, 5) circle (4pt);
\filldraw (0, 6) circle (4pt);
\filldraw (0, 7) circle (4pt);
\filldraw (0, 8) circle (4pt);
\filldraw (0, 9) circle (4pt);
\filldraw (0, 10) circle (4pt);
\filldraw (0, 11) circle (4pt);
\filldraw (0, 12) circle (4pt);
\filldraw (0, 13) circle (4pt);
\filldraw (0, 14) circle (4pt);
\filldraw (0, 15) circle (4pt);
\filldraw (0, 16) circle (4pt);
\filldraw (0, 17) circle (4pt);
\filldraw (0, 18) circle (4pt);
\filldraw (0, 19) circle (4pt);
\filldraw (0, 20) circle (4pt);
\filldraw (1, 1) circle (4pt);
\filldraw (1, 2) circle (4pt);
\filldraw (1, 3) circle (4pt);
\filldraw (1, 4) circle (4pt);
\filldraw (1, 5) circle (4pt);
\filldraw (1, 6) circle (4pt);
\filldraw (1, 7) circle (4pt);
\filldraw (1, 8) circle (4pt);
\filldraw (1, 9) circle (4pt);
\filldraw (1, 10) circle (4pt);

\filldraw (2, 0) circle (4pt);
\filldraw (2, 1) circle (4pt);
\filldraw (3, 0) circle (4pt);

\node[scale = 1.3, color = red] at (1, 11) {$\boldsymbol{\times}$};
\node[scale = 1.3, color = red] at (1, 12) {$\boldsymbol{\times}$};
\node[scale = 1.3, color = red] at (1, 13) {$\boldsymbol{\times}$};
\node[scale = 1.3, color = red] at (1, 14) {$\boldsymbol{\times}$};
\node[scale = 1.3, color = red] at (1, 15) {$\boldsymbol{\times}$};
\node[scale = 1.3, color = red] at (1, 16) {$\boldsymbol{\times}$};
\node[scale = 1.3, color = red] at (1, 17) {$\boldsymbol{\times}$};
\node[scale = 1.3, color = red] at (1, 18) {$\boldsymbol{\times}$};
\node[scale = 1.3, color = red] at (1, 19) {$\boldsymbol{\times}$};
\node[scale = 1.3, color = red] at (1, 20) {$\boldsymbol{\times}$};

\node[scale = .75, color = red] at (2, 10) {$\boldsymbol{\times}$};
\node[scale = .75, color = red] at (2, 11) {$\boldsymbol{\times}$};
\node[scale = .75, color = red] at (2, 12) {$\boldsymbol{\times}$};
\node[scale = .75, color = red] at (2, 13) {$\boldsymbol{\times}$};
\node[scale = .75, color = red] at (2, 14) {$\boldsymbol{\times}$};
\node[scale = .75, color = red] at (2, 15) {$\boldsymbol{\times}$};
\node[scale = .75, color = red] at (2, 16) {$\boldsymbol{\times}$};
\node[scale = .75, color = red] at (2, 17) {$\boldsymbol{\times}$};
\node[scale = .75, color = red] at (2, 18) {$\boldsymbol{\times}$};
\node[scale = .75, color = red] at (2, 19) {$\boldsymbol{\times}$};
\node[scale = 1.3, color = red] at (2, 20) {$\boldsymbol{\times}$};

\node[scale = .75, color = red] at (3, 10) {$\boldsymbol{\times}$};
\node[scale = .75, color = red] at (3, 11) {$\boldsymbol{\times}$};
\node[scale = .75, color = red] at (3, 12) {$\boldsymbol{\times}$};
\node[scale = .75, color = red] at (3, 13) {$\boldsymbol{\times}$};
\node[scale = .75, color = red] at (3, 14) {$\boldsymbol{\times}$};
\node[scale = .75, color = red] at (3, 15) {$\boldsymbol{\times}$};
\node[scale = .75, color = red] at (3, 16) {$\boldsymbol{\times}$};
\node[scale = .75, color = red] at (3, 17) {$\boldsymbol{\times}$};
\node[scale = 1.3, color = red] at (3, 18) {$\boldsymbol{\times}$};
\node[scale = 1.3, color = red] at (3, 19) {$\boldsymbol{\times}$};
\node[scale = 1.3, color = red] at (3, 20) {$\boldsymbol{\times}$};

\node[scale = .75, color = red] at (4, 9) {$\boldsymbol{\times}$};
\node[scale = .75, color = red] at (4, 10) {$\boldsymbol{\times}$};
\node[scale = .75, color = red] at (4, 11) {$\boldsymbol{\times}$};
\node[scale = .75, color = red] at (4, 12) {$\boldsymbol{\times}$};
\node[scale = .75, color = red] at (4, 13) {$\boldsymbol{\times}$};
\node[scale = .75, color = red] at (4, 14) {$\boldsymbol{\times}$};
\node[scale = .75, color = red] at (4, 15) {$\boldsymbol{\times}$};
\node[scale = 1.3, color = red] at (4, 16) {$\boldsymbol{\times}$};
\node[scale = 1.3, color = red] at (4, 17) {$\boldsymbol{\times}$};
\node[scale = 1.3, color = red] at (4, 18) {$\boldsymbol{\times}$};
\node[scale = 1.3, color = red] at (4, 19) {$\boldsymbol{\times}$};
\node[scale = 1.3, color = red] at (4, 20) {$\boldsymbol{\times}$};

\node[scale = .75, color = red] at (5, 9) {$\boldsymbol{\times}$};
\node[scale = .75, color = red] at (5, 10) {$\boldsymbol{\times}$};
\node[scale = .75, color = red] at (5, 11) {$\boldsymbol{\times}$};
\node[scale = .75, color = red] at (5, 12) {$\boldsymbol{\times}$};
\node[scale = .75, color = red] at (5, 13) {$\boldsymbol{\times}$};
\node[scale = 1.3, color = red] at (5, 14) {$\boldsymbol{\times}$};
\node[scale = 1.3, color = red] at (5, 15) {$\boldsymbol{\times}$};
\node[scale = 1.3, color = red] at (5, 16) {$\boldsymbol{\times}$};
\node[scale = 1.3, color = red] at (5, 17) {$\boldsymbol{\times}$};
\node[scale = 1.3, color = red] at (5, 18) {$\boldsymbol{\times}$};
\node[scale = 1.3, color = red] at (5, 19) {$\boldsymbol{\times}$};
\node[scale = 1.3, color = red] at (5, 20) {$\boldsymbol{\times}$};

\node[scale = .75, color = red] at (6, 9) {$\boldsymbol{\times}$};
\node[scale = .75, color = red] at (6, 10) {$\boldsymbol{\times}$};
\node[scale = .75, color = red] at (6, 11) {$\boldsymbol{\times}$};
\node[scale = 1.3, color = red] at (6, 12) {$\boldsymbol{\times}$};
\node[scale = 1.3, color = red] at (6, 13) {$\boldsymbol{\times}$};
\node[scale = 1.3, color = red] at (6, 14) {$\boldsymbol{\times}$};
\node[scale = 1.3, color = red] at (6, 15) {$\boldsymbol{\times}$};
\node[scale = 1.3, color = red] at (6, 16) {$\boldsymbol{\times}$};
\node[scale = 1.3, color = red] at (6, 17) {$\boldsymbol{\times}$};
\node[scale = 1.3, color = red] at (6, 18) {$\boldsymbol{\times}$};
\node[scale = 1.3, color = red] at (6, 19) {$\boldsymbol{\times}$};
\node[scale = 1.3, color = red] at (6, 20) {$\boldsymbol{\times}$};

\node[scale = .75, color = red] at (7, 8) {$\boldsymbol{\times}$};
\node[scale = .75, color = red] at (7, 9) {$\boldsymbol{\times}$};
\node[scale = 1.3, color = red] at (7, 10) {$\boldsymbol{\times}$};
\node[scale = 1.3, color = red] at (7, 11) {$\boldsymbol{\times}$};
\node[scale = 1.3, color = red] at (7, 12) {$\boldsymbol{\times}$};
\node[scale = 1.3, color = red] at (7, 13) {$\boldsymbol{\times}$};
\node[scale = 1.3, color = red] at (7, 14) {$\boldsymbol{\times}$};
\node[scale = 1.3, color = red] at (7, 15) {$\boldsymbol{\times}$};
\node[scale = 1.3, color = red] at (7, 16) {$\boldsymbol{\times}$};
\node[scale = 1.3, color = red] at (7, 17) {$\boldsymbol{\times}$};
\node[scale = 1.3, color = red] at (7, 18) {$\boldsymbol{\times}$};
\node[scale = 1.3, color = red] at (7, 19) {$\boldsymbol{\times}$};
\node[scale = 1.3, color = red] at (7, 20) {$\boldsymbol{\times}$};

\node[scale =  1.3, color = red] at (8, 8) {$\boldsymbol{\times}$};
\node[scale = 1.3, color = red] at (8, 9) {$\boldsymbol{\times}$};
\node[scale = 1.3, color = red] at (8, 10) {$\boldsymbol{\times}$};
\node[scale = 1.3, color = red] at (8, 11) {$\boldsymbol{\times}$};
\node[scale = 1.3, color = red] at (8, 12) {$\boldsymbol{\times}$};
\node[scale = 1.3, color = red] at (8, 13) {$\boldsymbol{\times}$};
\node[scale = 1.3, color = red] at (8, 14) {$\boldsymbol{\times}$};
\node[scale = 1.3, color = red] at (8, 15) {$\boldsymbol{\times}$};
\node[scale = 1.3, color = red] at (8, 16) {$\boldsymbol{\times}$};
\node[scale = 1.3, color = red] at (8, 17) {$\boldsymbol{\times}$};
\node[scale = 1.3, color = red] at (8, 18) {$\boldsymbol{\times}$};
\node[scale = 1.3, color = red] at (8, 19) {$\boldsymbol{\times}$};
\node[scale = 1.3, color = red] at (8, 20) {$\boldsymbol{\times}$};

\node[scale = 1.3, color = red] at (9, 6) {$\boldsymbol{\times}$};
\node[scale = 1.3, color = red] at (9, 7) {$\boldsymbol{\times}$};
\node[scale = 1.3, color = red] at (9, 8) {$\boldsymbol{\times}$};
\node[scale = 1.3, color = red] at (9, 9) {$\boldsymbol{\times}$};
\node[scale = 1.3, color = red] at (9, 10) {$\boldsymbol{\times}$};
\node[scale = 1.3, color = red] at (9, 11) {$\boldsymbol{\times}$};
\node[scale = 1.3, color = red] at (9, 12) {$\boldsymbol{\times}$};
\node[scale = 1.3, color = red] at (9, 13) {$\boldsymbol{\times}$};
\node[scale = 1.3, color = red] at (9, 14) {$\boldsymbol{\times}$};
\node[scale = 1.3, color = red] at (9, 15) {$\boldsymbol{\times}$};
\node[scale = 1.3, color = red] at (9, 16) {$\boldsymbol{\times}$};
\node[scale = 1.3, color = red] at (9, 17) {$\boldsymbol{\times}$};
\node[scale = 1.3, color = red] at (9, 18) {$\boldsymbol{\times}$};
\node[scale = 1.3, color = red] at (9, 19) {$\boldsymbol{\times}$};
\node[scale = 1.3, color = red] at (9, 20) {$\boldsymbol{\times}$};

\node[scale = 1.3, color = red] at (10, 4) {$\boldsymbol{\times}$};
\node[scale = 1.3, color = red] at (10, 5) {$\boldsymbol{\times}$};
\node[scale = 1.3, color = red] at (10, 6) {$\boldsymbol{\times}$};
\node[scale = 1.3, color = red] at (10, 7) {$\boldsymbol{\times}$};
\node[scale = 1.3, color = red] at (10, 8) {$\boldsymbol{\times}$};
\node[scale = 1.3, color = red] at (10, 9) {$\boldsymbol{\times}$};
\node[scale = 1.3, color = red] at (10, 10) {$\boldsymbol{\times}$};
\node[scale = 1.3, color = red] at (10, 11) {$\boldsymbol{\times}$};
\node[scale = 1.3, color = red] at (10, 12) {$\boldsymbol{\times}$};
\node[scale = 1.3, color = red] at (10, 13) {$\boldsymbol{\times}$};
\node[scale = 1.3, color = red] at (10, 14) {$\boldsymbol{\times}$};
\node[scale = 1.3, color = red] at (10, 15) {$\boldsymbol{\times}$};
\node[scale = 1.3, color = red] at (10, 16) {$\boldsymbol{\times}$};
\node[scale = 1.3, color = red] at (10, 17) {$\boldsymbol{\times}$};
\node[scale = 1.3, color = red] at (10, 18) {$\boldsymbol{\times}$};
\node[scale = 1.3, color = red] at (10, 19) {$\boldsymbol{\times}$};

\node[scale = 1.3, color = red] at (11, 2) {$\boldsymbol{\times}$};
\node[scale = 1.3, color = red] at (11, 3) {$\boldsymbol{\times}$};
\node[scale = 1.3, color = red] at (11, 4) {$\boldsymbol{\times}$};
\node[scale = 1.3, color = red] at (11, 5) {$\boldsymbol{\times}$};
\node[scale = 1.3, color = red] at (11, 6) {$\boldsymbol{\times}$};
\node[scale = 1.3, color = red] at (11, 7) {$\boldsymbol{\times}$};
\node[scale = 1.3, color = red] at (11, 8) {$\boldsymbol{\times}$};
\node[scale = 1.3, color = red] at (11, 9) {$\boldsymbol{\times}$};
\node[scale = 1.3, color = red] at (11, 10) {$\boldsymbol{\times}$};
\node[scale = 1.3, color = red] at (11, 11) {$\boldsymbol{\times}$};
\node[scale = 1.3, color = red] at (11, 12) {$\boldsymbol{\times}$};
\node[scale = 1.3, color = red] at (11, 13) {$\boldsymbol{\times}$};
\node[scale = 1.3, color = red] at (11, 14) {$\boldsymbol{\times}$};
\node[scale = 1.3, color = red] at (11, 15) {$\boldsymbol{\times}$};
\node[scale = 1.3, color = red] at (11, 16) {$\boldsymbol{\times}$};
\node[scale = 1.3, color = red] at (11, 17) {$\boldsymbol{\times}$};
\node[scale = 1.3, color = red] at (11, 18) {$\boldsymbol{\times}$};

\node[scale = 1.3, color = red] at (12, 0) {$\boldsymbol{\times}$};
\node[scale = 1.3, color = red] at (12, 1) {$\boldsymbol{\times}$};
\node[scale = 1.3, color = red] at (12, 2) {$\boldsymbol{\times}$};
\node[scale = 1.3, color = red] at (12, 3) {$\boldsymbol{\times}$};
\node[scale = 1.3, color = red] at (12, 4) {$\boldsymbol{\times}$};
\node[scale = 1.3, color = red] at (12, 5) {$\boldsymbol{\times}$};
\node[scale = 1.3, color = red] at (12, 6) {$\boldsymbol{\times}$};
\node[scale = 1.3, color = red] at (12, 7) {$\boldsymbol{\times}$};
\node[scale = 1.3, color = red] at (12, 8) {$\boldsymbol{\times}$};
\node[scale = 1.3, color = red] at (12, 9) {$\boldsymbol{\times}$};
\node[scale = 1.3, color = red] at (12, 10) {$\boldsymbol{\times}$};
\node[scale = 1.3, color = red] at (12, 11) {$\boldsymbol{\times}$};
\node[scale = 1.3, color = red] at (12, 12) {$\boldsymbol{\times}$};
\node[scale = 1.3, color = red] at (12, 13) {$\boldsymbol{\times}$};
\node[scale = 1.3, color = red] at (12, 14) {$\boldsymbol{\times}$};
\node[scale = 1.3, color = red] at (12, 15) {$\boldsymbol{\times}$};
\node[scale = 1.3, color = red] at (12, 16) {$\boldsymbol{\times}$};
\node[scale = 1.3, color = red] at (12, 17) {$\boldsymbol{\times}$};

\node[scale = .75, color = red] at (13, 0) {$\boldsymbol{\times}$};
\node[scale = .75, color = red] at (13, 1) {$\boldsymbol{\times}$};
\node[scale = .75, color = red] at (13, 2) {$\boldsymbol{\times}$};
\node[scale = .75, color = red] at (13, 3) {$\boldsymbol{\times}$};
\node[scale = .75, color = red] at (13, 4) {$\boldsymbol{\times}$};
\node[scale = .75, color = red] at (13, 5) {$\boldsymbol{\times}$};
\node[scale = 1.3, color = red] at (13, 6) {$\boldsymbol{\times}$};
\node[scale = 1.3, color = red] at (13, 7) {$\boldsymbol{\times}$};
\node[scale = 1.3, color = red] at (13, 8) {$\boldsymbol{\times}$};
\node[scale = 1.3, color = red] at (13, 9) {$\boldsymbol{\times}$};
\node[scale = 1.3, color = red] at (13, 10) {$\boldsymbol{\times}$};
\node[scale = 1.3, color = red] at (13, 11) {$\boldsymbol{\times}$};
\node[scale = 1.3, color = red] at (13, 12) {$\boldsymbol{\times}$};
\node[scale = 1.3, color = red] at (13, 13) {$\boldsymbol{\times}$};
\node[scale = 1.3, color = red] at (13, 14) {$\boldsymbol{\times}$};
\node[scale = 1.3, color = red] at (13, 15) {$\boldsymbol{\times}$};
\node[scale = 1.3, color = red] at (13, 16) {$\boldsymbol{\times}$};

\node[scale = .75, color = red] at (14, 0) {$\boldsymbol{\times}$};
\node[scale = .75, color = red] at (14, 1) {$\boldsymbol{\times}$};
\node[scale = .75, color = red] at (14, 2) {$\boldsymbol{\times}$};
\node[scale = .75, color = red] at (14, 3) {$\boldsymbol{\times}$};
\node[scale = 1.3, color = red] at (14, 4) {$\boldsymbol{\times}$};
\node[scale =  1.3, color = red] at (14, 5) {$\boldsymbol{\times}$};
\node[scale = 1.3, color = red] at (14, 6) {$\boldsymbol{\times}$};
\node[scale = 1.3, color = red] at (14, 7) {$\boldsymbol{\times}$};
\node[scale = 1.3, color = red] at (14, 8) {$\boldsymbol{\times}$};
\node[scale = 1.3, color = red] at (14, 9) {$\boldsymbol{\times}$};
\node[scale = 1.3, color = red] at (14, 10) {$\boldsymbol{\times}$};
\node[scale = 1.3, color = red] at (14, 11) {$\boldsymbol{\times}$};
\node[scale = 1.3, color = red] at (14, 12) {$\boldsymbol{\times}$};
\node[scale = 1.3, color = red] at (14, 13) {$\boldsymbol{\times}$};
\node[scale = 1.3, color = red] at (14, 14) {$\boldsymbol{\times}$};
\node[scale = 1.3, color = red] at (14, 15) {$\boldsymbol{\times}$};

\node[scale = .75, color = red] at (15, 0) {$\boldsymbol{\times}$};
\node[scale = .75, color = red] at (15, 1) {$\boldsymbol{\times}$};
\node[scale = 1.3, color = red] at (15, 2) {$\boldsymbol{\times}$};
\node[scale = 1.3, color = red] at (15, 3) {$\boldsymbol{\times}$};
\node[scale = 1.3, color = red] at (15, 4) {$\boldsymbol{\times}$};
\node[scale =  1.3, color = red] at (15, 5) {$\boldsymbol{\times}$};
\node[scale = 1.3, color = red] at (15, 6) {$\boldsymbol{\times}$};
\node[scale = 1.3, color = red] at (15, 7) {$\boldsymbol{\times}$};
\node[scale = 1.3, color = red] at (15, 8) {$\boldsymbol{\times}$};
\node[scale = 1.3, color = red] at (15, 9) {$\boldsymbol{\times}$};
\node[scale = 1.3, color = red] at (15, 10) {$\boldsymbol{\times}$};
\node[scale = 1.3, color = red] at (15, 11) {$\boldsymbol{\times}$};
\node[scale = 1.3, color = red] at (15, 12) {$\boldsymbol{\times}$};
\node[scale = 1.3, color = red] at (15, 13) {$\boldsymbol{\times}$};
\node[scale = 1.3, color = red] at (15, 14) {$\boldsymbol{\times}$};

\node[scale = 1.3, color = red] at (16, 0) {$\boldsymbol{\times}$};
\node[scale = 1.3, color = red] at (16, 1) {$\boldsymbol{\times}$};
\node[scale = 1.3, color = red] at (16, 2) {$\boldsymbol{\times}$};
\node[scale = 1.3, color = red] at (16, 3) {$\boldsymbol{\times}$};
\node[scale = 1.3, color = red] at (16, 4) {$\boldsymbol{\times}$};
\node[scale =  1.3, color = red] at (16, 5) {$\boldsymbol{\times}$};
\node[scale = 1.3, color = red] at (16, 6) {$\boldsymbol{\times}$};
\node[scale = 1.3, color = red] at (16, 7) {$\boldsymbol{\times}$};
\node[scale = 1.3, color = red] at (16, 8) {$\boldsymbol{\times}$};
\node[scale = 1.3, color = red] at (16, 9) {$\boldsymbol{\times}$};
\node[scale = 1.3, color = red] at (16, 10) {$\boldsymbol{\times}$};
\node[scale = 1.3, color = red] at (16, 11) {$\boldsymbol{\times}$};
\node[scale = 1.3, color = red] at (16, 12) {$\boldsymbol{\times}$};
\node[scale = 1.3, color = red] at (16, 13) {$\boldsymbol{\times}$};

\node[scale = 1.3, color = red] at (17, 0) {$\boldsymbol{\times}$};
\node[scale = 1.3, color = red] at (17, 1) {$\boldsymbol{\times}$};
\node[scale = 1.3, color = red] at (17, 2) {$\boldsymbol{\times}$};
\node[scale = 1.3, color = red] at (17, 3) {$\boldsymbol{\times}$};
\node[scale = 1.3, color = red] at (17, 4) {$\boldsymbol{\times}$};
\node[scale =  1.3, color = red] at (17, 5) {$\boldsymbol{\times}$};
\node[scale = 1.3, color = red] at (17, 6) {$\boldsymbol{\times}$};
\node[scale = 1.3, color = red] at (17, 7) {$\boldsymbol{\times}$};
\node[scale = 1.3, color = red] at (17, 8) {$\boldsymbol{\times}$};
\node[scale = 1.3, color = red] at (17, 9) {$\boldsymbol{\times}$};
\node[scale = 1.3, color = red] at (17, 10) {$\boldsymbol{\times}$};
\node[scale = 1.3, color = red] at (17, 11) {$\boldsymbol{\times}$};
\node[scale = 1.3, color = red] at (17, 12) {$\boldsymbol{\times}$};

\draw[color=red, ultra thick, ->] (13, 17) -- (14.5, 18.5);

\filldraw[color=black] (2,2) circle (4pt);
\filldraw[color=black] (2,3) circle (4pt);
\filldraw[color=black] (2,4) circle (4pt);
\filldraw[color=black] (2,5) circle (4pt);
\filldraw[color=black] (2,6) circle (4pt);
\filldraw[color=black] (2,7) circle (4pt);
\filldraw[color=black] (2,8) circle (4pt);
\filldraw[color=black] (2,9) circle (4pt);
\draw[color=black,ultra thick] (2,10) circle (4pt);

\filldraw[color=black] (3,1) circle (4pt);
\filldraw[color=black] (3,2) circle (4pt);
\filldraw[color=black] (3,3) circle (4pt);
\filldraw[color=black] (3,4) circle (4pt);
\filldraw[color=black] (3,5) circle (4pt);
\filldraw[color=black] (3,6) circle (4pt);
\filldraw[color=black] (3,7) circle (4pt);
\filldraw[color=black] (3,8) circle (4pt);
\draw[color=black,ultra thick] (3,9) circle (4pt);
\draw[color=black,ultra thick] (3,10) circle (4pt);
\draw[color=black,ultra thick] (3,11) circle (4pt);

\filldraw[color=black] (4,0) circle (4pt);
\filldraw[color=black] (4,1) circle (4pt);
\filldraw[color=black] (4,2) circle (4pt);
\filldraw[color=black] (4,3) circle (4pt);
\filldraw[color=black] (4,4) circle (4pt);
\filldraw[color=black] (4,5) circle (4pt);
\filldraw[color=black] (4,6) circle (4pt);
\draw[color=black,ultra thick] (4,7) circle (4pt);
\draw[color=black,ultra thick] (4,8) circle (4pt);
\draw[color=black,ultra thick] (4,9) circle (4pt);
\draw[color=black,ultra thick] (4,10) circle (4pt);
\draw[color=black,ultra thick] (4,11) circle (4pt);

\filldraw[color=black] (5,0) circle (4pt);
\filldraw[color=black] (5,1) circle (4pt);
\filldraw[color=black] (5,2) circle (4pt);
\filldraw[color=black] (5,3) circle (4pt);
\filldraw[color=black] (5,4) circle (4pt);
\draw[color=black,ultra thick] (5,5) circle (4pt);
\draw[color=black,ultra thick] (5,6) circle (4pt);
\draw[color=black,ultra thick] (5,7) circle (4pt);
\draw[color=black,ultra thick] (5,8) circle (4pt);
\draw[color=black,ultra thick] (5,9) circle (4pt);

\filldraw[color=black] (6,0) circle (4pt);
\filldraw[color=black] (6,1) circle (4pt);
\filldraw[color=black] (6,2) circle (4pt);
\draw[color=black,ultra thick] (6,3) circle (4pt);
\draw[color=black,ultra thick] (6,4) circle (4pt);
\draw[color=black,ultra thick] (6,5) circle (4pt);

\filldraw[color=black] (7,0) circle (4pt);

\filldraw[color=white] (8, 0) circle (4pt);
\filldraw[color=white] (9, 0) circle (4pt);
\draw[ultra thick] (8, 0) circle (4pt);
\draw[ultra thick] (9, 0) circle (4pt);

\end{tikzpicture}
\end{center}
\end{figure}

The implications next to the symbol key are from Lemma \ref{easy} and its contrapositive.
When the chart is completed, every position will be filled with exactly one of the following symbols

\begin{center}
\begin{tikzpicture}[scale=.8]
\filldraw (0, 0) circle (4pt);
\node[scale = .75, color = red] at (4, 0) {$\boldsymbol{\times}$};
\draw[ultra thick] (4,0) circle (4pt);
\node[scale = 1.3, color = red] at (8, 0) {$\boldsymbol{\times}$};
\end{tikzpicture}
\end{center}

\begin{rem} \label{rem:ckgp}
Currently, in every case where a filled circle appears in the chart, $\Mb_{g,n}$ is known to have the CKgP. Similarly, in every instance where an open circle appears in the chart, $\M_{g,n}$ is known to have the CKgP.
\end{rem}

\subsection{Genus 0}
In genus $0$, we have that $\M_{0,n}$ is an open subset of affine space, so the Chow ring is trivial $A^*(\M_{0,n})  = \qq$. It follows quickly from excision and the structure of the boundary that the Chow groups
$A^*(\Mb_{0,n})$ are generated by fundamental classes of boundary strata and are hence tautological. The relations among boundary strata and the ring structure of $A^*(\Mb_{0,n})$ is computed in work of Keel \cite{Keel}.

\subsection{Genus 1}
The fact that $A^*(\Mb_{1,n}) = R^*(\Mb_{1,n})$ for $n \leq 10$ was established in the thesis of Belorousski \cite{Bel}. Meanwhile, it is known that $H^{11}(\Mb_{1,11}) \neq 0$. This means that $H^*(\Mb_{1,11}) \neq RH^*(\Mb_{1,11})$ and so by Corollary \ref{rhcor}, $A^*(\Mb_{1,11}) \neq R^*(\Mb_{1,11})$. Moreover, the presence of odd cohomology means that the cycle class map cannot be an isomorphism, so by Lemma \ref{cl_iso}, we conclude that $A^*(\Mb_{1,11})$ is an uncountably infinite dimensional vector space over $\qq$. 

\begin{rem}
Going further, it is known that $H^{11,0}(\Mb_{1,11}) \neq 0$, see \cite[p. 306]{FP} for an explicit construction of a holomorphic form.
On a smooth, proper variety $X$, the presence of a holomorphic $p$ form for $p > 1$ implies that the Chow group of zero cycles is so large it cannot even be dominated by a finite-dimensional variety.
References and an explanation are given in \cite[Remark 1.1]{GraberVakil} for why the same conclusion holds for the smooth, proper Deligne--Mumford stack $\Mb_{1,11}$.
\end{rem}

Because the push forwards along forgetful maps are surjective, the fact that
$A^*(\Mb_{1,11})$ is infinite-dimensional implies that $A^*(\Mb_{1,n})$ is infinite-dimensional for all $n \geq 11$.
Some additional care is needed to argue that the Chow rings of the open moduli spaces are also infinite-dimensional, thus giving the large red exes in the rest of the genus $1$ column.

\begin{lem}
For all $n \geq 11$, we have $A^*(\M_{1,n})$ is an infinite-dimensional vector space. In particular, $A^*(\M_{1,n}) \neq R^*(\M_{1,n})$.
\end{lem}
\begin{proof}
We first prove the claim for $n = 11$. We know that $A^*(\Mb_{1,11})$ is infinite-dimensional, so by excision, it suffices to show that $A^*(\partial \M_{1,11})$ is finite-dimensional. The boundary 
of $\Mb_{1,11}$ is the image of gluing maps from products of genus $0$ moduli spaces and genus $1$ moduli spaces with at most $10$ markings. All genus $0$ moduli spaces have the CKgP, we have the following sequence of surjective maps
\begin{center}
\begin{tikzcd}
A^*(\Mb_{0,13}) \oplus  \bigoplus_{A \subset \{1,\ldots, 11\}} A^*(\Mb_{1,A \cup p}) \otimes A^*(\Mb_{0,A^c \cup p'}) \arrow{d} \\
A^*(\Mb_{0,13}) \oplus  \bigoplus_{A \subset \{1,\ldots, 11\}} A^*(\Mb_{1,A \cup p} \times \Mb_{0,A^c \cup p'}) \arrow{d} \\
A^*(\partial \M_{1,11}). 
\end{tikzcd}
\end{center}
In the sum in the top row, we sum over $A$ such that $|A^c| \geq 2$ so that the genus $0$ curve that is glued to the genus $1$ curve is stable.
Thus, $|A \cup p| \leq 10$, so  $A^*(\Mb_{1,A \cup p}) = R^*(\Mb_{1,A \cup p})$ is finite-dimensional. Hence, the 
vector space in the top row above is finite-dimensional. It follows that $A^*(\partial \M_{1,11})$ is finite-dimensional, which completes the claim for $n = 11$.

Now assume for contradiction that the lemma is false. Then, there exists some $n \geq 11$ such that $A^*(\M_{1,n})$ is an infinite-dimensional vector space but $A^*(\M_{1,n+1})$ is finite-dimensional.
Because $A^*(\M_{1,n})$ is infinite-dimensional, so is $A^*(\Mb_{1,n})$. Let $i$ be the minimal index such that $A^i(\Mb_{1,n})$ is infinite-dimensional. Because $\Mb_{1,n+1} \to \Mb_{1,n}$ has a section, the pullback map is injective. Hence, $A^i(\Mb_{1,n+1})$ is infinite-dimensional. By excision, it suffices to show that $A^{i-1}(\partial \M_{1,n+1})$ is finite-dimensional, as this will contradict the assumption that $A^*(\M_{1,n+1})$ is finite-dimensional.

Because all $\Mb_{0,m}$ have the CKgP, we have a sequence of surjective maps
\begin{center}
\begin{tikzcd}
A^{i-1}(\Mb_{0,n+3}) \oplus  \bigoplus_{A \subset \{1,\ldots, n+1\}} \bigoplus_{j+k = i-1} A^j(\Mb_{1,A \cup p}) \otimes A^k(\Mb_{0,A^c \cup p'}) \arrow{d} \\
A^{i-1}(\Mb_{0,n+3}) \oplus  \bigoplus_{A \subset \{1,\ldots, n+1\}} A^{i-1}(\Mb_{1,A \cup p} \times \Mb_{0,A^c \cup p'}) \arrow{d} \\
A^{i-1}(\partial \M_{1,n+1}). 
\end{tikzcd}
\end{center}
Here, the sum runs over $A$ such that $|A^c| \geq 2$, so $|A \cup p| \leq n$.
By the minimality of $i$, we know that $A^j(\Mb_{1,n})$ is finite-dimensional for all $j \leq i - 1$. 
Because pullbacks along the forgetful maps are injective, it follows that $A^j(\Mb_{1,A \cup p})$ is finite dimensional for all $|A \cup p| \leq n$ and $j \leq i - 1$. 
Thus, the vector space at the top of the previous display is finite-dimensional. We conclude that $A^{i-1}(\partial \M_{1,n+1})$ is finite-dimensional, as desired.
\end{proof}

Since the Chow groups are so large, there are plenty of non-tautological classes in $A^*(\Mb_{1,n})$. Nevertheless, the following result implies that their images under the cycle class map are all tautological in cohomology.

\begin{thm}[Petersen \cite{Pg1}] \label{g1}
We have $H^{2k}(\Mb_{1,n}) = RH^{2k}(\Mb_{1,n})$, which is spanned by fundamental classes of boundary strata.
\end{thm}

\subsection{Tautological classes in genus 2 through 9} \label{29}
The filled and open circles in genus $2$ through $9$ were established  \cite{ckgp} with the following exceptions (some of the open circles referenced below were established earlier, and are then covered in the chart with a filled circle from \cite{ckgp}):
\begin{itemize}
\item $g = 2$ and $n = 0$ by Mumford \cite{M}, $g = 2$ and $n = 1$ by Faber \cite{FaberThesis}
\item $g = 3$ and $n = 0$ by Faber \cite{F2}
\item[${\boldmath{\circ}}$] $g = 4$ and $n = 0$ by Faber \cite{F3}
\item [${\boldmath{\circ}}$] $g = 5$ and $n = 0$ by Izadi \cite{Iz}
\item[${\boldmath{\circ}}$] $g = 5$ and $n = 8, 9$ by Liu \cite{Yuhan}
\item [${\boldmath{\circ}}$] $g = 6$ and $n = 0$ by Penev--Vakil \cite{PV}
\item[${\boldmath{\circ}}$] $g = 7$ and $1 \leq n \leq 3$ by  Canning--Larson--Payne \cite{ste}
\item[${\boldmath{\circ}}$] $g = 7, 8, 9$ and $n=0$ by Canning--Larson \cite{789}
\end{itemize}

Each of these works involves explicitly constructing stratifications of $\M_{g,n}$ and then determining the Chow rings of each stratum by using convenient models of the pointed curve. For example, in genus $3$, we stratify $\M_{3,n}$ into the hyperelliptic and non-hyperelliptic curves. The non-hyperelliptic curves are smooth plane quartics. It turns out that if $n \leq 11$, then $n$ points lying on a smooth plane quartic always impose independent conditions on smooth plane quartics. In this way, the non-hyperelliptic locus of $\M_{3,n}$ is realized as an open substack of a vector bundle over an open substack of the moduli space of $n$ distinct points on $\pp^2$. This sort of description furnishes generators for the Chow ring of the stratum. Finally, one checks that these generators are tautological.

The case $g = 5$ and $n = 8,9$ is particularly interesting. As above, one first separates out the curves of gonality $2$ and $3$.
Then, the tetragonal curves can be modeled as the complete intersection of three quadrics in $\pp^3$. If $n \leq 7$, then in \cite{ckgp}, it was shown that any $n$ points on a canonical genus $5$ curve impose independent conditions on quadrics in $\pp^3$. However, this breaks down when $n \geq 8$. The innovation of Liu in \cite{Yuhan} is to prove that when $n = 8,9$ the points fail to impose independent conditions on quadrics if and only if $8$ of them lie in a hyperplane. Liu then stratifies the tetragonal locus according to whether or not the $n$ points impose independent conditions on quadrics, realizing each piece as an open substack of a vector bundle over a better understood moduli space of points.

Once the Chow ring of $\M_{g,n}$ is shown to be tautological, one can try to assemble the pieces of the boundary using Lemma \ref{assemble}. However, it is not necessary that $A^*(\M_{g',n'}) = R^*(\M_{g',n'})$ for all $2g' + n' \leq 2g + n$.  
Indeed, one interesting example is that, although $A^*(\Mb_{1,11})$ is infinite-dimensional, we have
 $A^*(\Mb_{2,9}) = R^*(\Mb_{2,9})$. This is proved in \cite{ckgp} by modeling pointed genus $2$ curves on $\pp^1 \times \pp^1$ and allowing for the possibility that the curves are simply nodal. In this way, we gain a handle on the Chow ring of $\Mb_{2,9}$ minus the boundary divisors with a separating node, and show it is generated by tautological classes. Then, the boundary divisors with a separating node are the images of 
gluing maps whose sources are products of smaller moduli spaces with the CKgP whose Chow rings are known to be tautological.

\subsection{Non-tautological classes} \label{non-taut}
Outside of genus $1$, the main known source of non-tautological classes are loci of double covers. Let
\[\B_{g,2m} = \{C \in \M_g : \exists \ \pi: C \to E \text{ for $E$ a genus $1$ curve and $\pi(p_{2i}) = \pi(p_{2i-1})$}\]
be the locus of pointed bielliptic curves where the marked points are constrained to be conjugate under the degree $2$ cover. Write $\overline{\B}_{g,2m}$ for the closure of $\B_{g,2m}$ in $\Mb_{g,2m}$.

\begin{thm}[van Zelm \cite{VZ}] \label{vzx}
For $g + m \geq 12$, we have $[\overline{\B}_{g,2m}] \notin R^*(\Mb_{g,2m})$. Its image under the cycle class map is also non-tautological.
\end{thm}

In the case $g +m = 12$, the method for showing these cycles are non-tautological involves pulling back along appropriate gluing maps $\Mb_{1,11} \times \Mb_{1,11} \to \Mb_{g,2m}$. The pullback of $\overline{\B}_{g,2m}$ is expressible as a multiple of the diagonal plus some contribution from the boundary. The fact that $H^{11}(\Mb_{1,11}) \neq 0$ implies that the diagonal inside $\Mb_{1,11} \times \Mb_{1,11}$
has non-tautological K\"unneth components. However, the pullbacks of tautological classes always have tautological K\"unneth comonents by \cite[Appendix A]{GP}. 
Because pushforward along forgetful maps is surjective on Chow groups and sends tautological classes to tautological classes, if $A^*(\Mb_{g,n}) \neq R^*(\Mb_{g,n})$, then $A^*(\Mb_{g,n+1}) \neq R^*(\Mb_{g,n+1})$. Thus, we can draw small red exes in the region $2g + n \geq 24$.

Additional care is needed to show that these bielliptic cycles yield non-tautological classes on the interior $\M_{g,n}$, represented by large red exes.
These are a consequence of the following theorem.
\begin{thm}[see \cite{GP,VZ,cc,FT}] \label{mainx}
There exists a non-tautological class in $A^*(\M_{g,n})$ whose image in $H^*(\M_{g,n})$ is also non-tautological for $g$ and $n$ satisfying either of the following conditions
\begin{enumerate}
\item[{\color{red} \large $\times$}] $2 \leq g \leq 12$ and $2g + n \geq 24$
\item[{\color{red} \large $\times$}] $g \geq 13$ and $2g + n \geq 32$.
\end{enumerate}
\end{thm}

In each of these cases, it is shown that the fundamental class of a locus of double covers of curves is non-tautological.
The first case, treated in \cite{GP} was $(g, n) = (2,20)$. 
The cases with $2 \leq g \leq 12$ and $2g + n = 24$ were established by van Zelm \cite{VZ}. 
In these cases, the non-tautological cycle on $\M_{g,n}$ is $\B_{g,n}$.
In \cite{cc}, these ideas were generalized to double covers of curves of genus $> 1$, establishing the cases with $g \geq 4$ and $n$ is even and $2g + n \geq 32$, as well as $g = 2, 3$ and $n$ is even and $4 \mid 2g + n \geq 32$. The remaining cases, especially with $n$ odd, are treated by improving upon these ideas in \cite{FT}.

\smallskip
Finally, the small red exes, seen in the chart with $g < 12$, are cases where $H^k(\Mb_{g,n}) \neq 0$ for some odd $k$, which by Corollary \ref{rhcor}, implies that $A^*(\Mb_{g,n}) \neq R^*(\Mb_{g,n}$).
These odd cohomology classes come from pushing forward classes in $H^{11}(\Mb_{1,11})$ along gluing maps. One can verify that these push forwards are non-zero by pulling them back along the same gluing map and using the push-pull formula.
\begin{thm}[Canning--L.--Payne \cite{h11}] \label{odd1}
Assume $k \leq 11$. Let $g_1,\ldots, g_k$ be distinct positive integers, and set $g = 1 + g_1 + \ldots + g_k$. Then
\[H^{11+2k}(\Mb_{g,n}) \neq 0 \text{ for $n \geq 11 - k$}.\]
\end{thm}

The above result implies $H^{33}(\Mb_{g}) \neq 0$ for $g \geq 67$.
We have also used similar ideas to construct odd cohomology classes starting with known non-zero classes in $H^{17}(\Mb_{2,14})$. This produces odd cohomology classes in genus $16$ with no markings.

\begin{thm}[Canning--L.--Payne--Willwacher \cite{FAmodules}] \label{odd2}
For $g \geq 16$ and $n \geq 0$, we have $H^{45+2(g - 16)}(\Mb_{g,n}) \neq 0$.
\end{thm}

\subsection{Interpreting the chart in cohomology}
Considering the references we have given in Section \ref{non-taut}, every red ex in the chart is also an instance where $H^*(\Mb_{g,n}) \neq RH^*(\Mb_{g,n})$.
Moreover, by Corollary \ref{rhcor}, every filled circle in the chart is an instance where $H^*(\Mb_{g,n}) = RH^*(\Mb_{g,n})$. 
There are a few further instances where all cohomology is known to be tautological. For the pairs $(g, n)$ listed in the proposition below, this is a slightly stronger statement than in \cite[Theorem 8]{C}, where it is shown that all even degree cohomology is tautological. We defer the proof to the end of Section \ref{sec:pts} as it weaves together several results explained in the next two sections.
\begin{prop} \label{samir8}
We have $H^*(\Mb_{g,n}) = RH^*(\Mb_{g,n})$ for the following pairs of $(g, n)$:
\[(3, 9) \qquad (4, 7) \qquad (5,5) \qquad (6,3) \qquad (7,1).\]
\end{prop}

\begin{rem}
Although the references given in Section \ref{29} for $(g,n) = (2,0), (2,1), (3, 0)$ provide a presentation of the ring, in general even when
 $H^*(\Mb_{g,n}) = RH^*(\Mb_{g,n})$, we do not necessarily understand the cohomology ring, since the tautological ring is not known. 
However,
the fact that $H^*(\Mb_{g,n}) = RH^*(\Mb_{g,n})$ together with Poincar\'e duality implies that the pairing in the tautological ring is perfect. Moreover, there are known algorithms to compute the pairing in the tautological ring, which would then determine all relations (a class is zero if and only if it pairs to zero with all classes in complimentary degree). Such calculations are very computationally expensive, but progress in this direction, in particular determining the Betti numbers when $n = 0$ and $g \leq 6$, is the subject of forthcoming work \cite{S+}.
\end{rem}

\begin{rem}
The reader is warned that although open circles imply the \emph{pure weight} cohomology of $\M_{g,n}$ is tautological by Remark \ref{rem:ckgp} and Lemma \ref{open-ckgp}, they do not imply $H^*(\M_{g,n}) = RH^*(\M_{g,n})$. For example $H^*(\M_{0,4}) \neq RH^*(\M_{0,4})$.
\end{rem}

\subsection{The Madsen--Weiss theorem} We end this section with a celebrated result about the cohomology of the interior $\M_{g,n}$ and its tautological ring. Although we rarely expect the entire cohomology ring of $\M_{g,n}$ to be tautological, the Madsen--Weiss theorem says that the cohomology groups in low degrees relative to $g$ are all tautological. The extension to marked points is due to Looijenga. The precise range where the cohomology groups become tautological has improved over the years. The statement we give below can be found in \cite[Theorem 1.1]{Wahl}.
Recall that by excision and our definition of the tautological ring in the introduction, $R^*(\M_{g,n})$ is generated by the psi and kappa classes. Thus, $RH^*(\M_{g,n})$ is generated by their images under the cycle class map.

\begin{thm}[Harer, Madsen--Weiss, Looijenga]
Let $\qq[\kappa_1, \kappa_2, \ldots, \psi_1, \ldots, \psi_n]$ be the free polynomial algebra with generators $\psi_1, \ldots, \psi_n$ in degree $2$ and $\kappa_j$ in degree $2j$. The map
\[ \qq[\kappa_1, \kappa_2, \ldots, \psi_1, \ldots, \psi_n] \to H^*(\M_{g,n})\]
is an isomorphism  in degrees $* \leq \frac{2}{3}g - \frac{2}{3}$.
\end{thm}

\section{Cohomology in low weights} \label{sec:coh}

As we have seen in the previous section, for $g > 0$, there are only finitely many moduli spaces $\Mb_{g,n}$ whose entire cohomology ring is tautological. Nevertheless, we may restrict to a given cohomological degree $k$ and ask: Is $H^k(\Mb_{g,n})$ tautological, and if not, can we give generators?  
The first results in this direction are due to Arbarello and Cornalba in the 1990s.

\begin{thm}[Arbarello--Cornalba \cite{ArbarelloCornalba}] \label{ac2}
Let $k = 0, 1, 2, 3, 5$. Then $H^k(\Mb_{g,n}) = RH^k(\Mb_{g,n})$ for all $g$ and $n$.
\end{thm}

The tautological ring vanishes in odd degrees, so
when $k$ is odd, the theorem simply says that $H^k(\Mb_{g,n}) = 0$. The most difficult case of the theorem is when $k = 2$. In this case, Arbarello and Cornalba provide a basis for $RH^2(\Mb_{g,n})$. For $g \geq 3$, there are no relations among the generators $\kappa_1, \psi_1, \ldots, \psi_n$, and the fundamental classes of boundary strata. For $g \leq 2$, explicit relations are given, see \cite[Theorem 2.2]{ArbarelloCornalba}.

\medskip
Arbarello and Cornalba's proof utilizes pullbacks to the boundary to set up an inductive argument. The essential lemma is the following.

\begin{lem}[Lemma 2.5 \cite{ArbarelloCornalba}, Proposition 2.1 \cite{BFP}] \label{ind}
Define 
\[d(g,n) = \begin{cases} n - 4 & \text{if $g = 0$} \\ 2g - 2 + n & \text{if $g>0$ and $n \geq 2$} \\
2g. &\text{if $g > 0$ and $n = 0,1$.}\end{cases}\]
If $d(g,n) > k$, then the pullback map $H^k(\Mb_{g,n}) \to H^k(\widetilde{\partial \M_{g,n}})$ is injective.
\end{lem}

To see the utility of the lemma, we sketch how it can be used prove the vanishing of low-degree odd cohomology groups inductively. Fix some odd integer $k$. Suppose that we know $H^{k'}(\Mb_{g',n'}) = 0$ for odd $k' \leq k$ and the finitely many $g'$ and $n'$ with $d(g',n') \leq k$. These are the base cases of the induction. Now suppose we are given some $g, n$ with $d(g,n) > k$. By Lemma \ref{ind}, we have an injection
\begin{align*}
H^k(\Mb_{g,n}) \hookrightarrow \ &H^k(\widetilde{\partial \M_{g,n}}) = H^k(\Mb_{g-1,n+2}) \oplus \bigoplus_{\substack{A_1 \sqcup A_2 = \{1, \ldots, n\} \\ g_1 + g_2 = g }} H^{k}(\Mb_{g_1,A_1\cup p} \times \Mb_{g_2, A_2 \cup p'})
 \\
&=H^k(\Mb_{g-1,n+2}) \oplus \bigoplus_{\substack{A_1 \sqcup A_2= \{1, \ldots, n\} \\ g_1 + g_2 = g }} \bigoplus_{k_1+k_2 = k} H^{k_1}(\Mb_{g_1,A_1\cup p}) \otimes H^{k_2}(\Mb_{g_2, A_2 \cup p'})
\end{align*}
where we have used the K\"unneth formula in the second row. By induction on $g$, we may assume $H^k(\Mb_{g-1,n+2}) = 0$.
Since $k$ is odd, one of $k_1$ or $k_2$ is odd, say $k_i$. Additionally, either $g_i < g$ or if $g_i = g$ then $|A_i| + 1 < n$. Thus, by induction, we may assume that $H^{k_i}(\Mb_{g_i,A_i\cup p}) = 0$. This shows that the right-hand side vanishes. Since the map is injective, we conclude that $H^k(\Mb_{g,n}) = 0$ for all $k$.

\medskip
In order to improve upon Arbarello and Cornalba's results and run this strategy for odd values of $k$ greater than $5$, more base cases were needed, particularly the cases $g = 4$ and $n \leq 3$. In \cite{BFP}, Bergstr\"om--Faber--Payne proved the required odd cohomology vanishing for these base cases via point counting. They use models of pointed genus $4$ cures as complete intersections of a quadric and a cubic in $\pp^3$ and an intricate sieve to count the smooth such intersections. In the end, they find that point count is given by a polynomial in $q$. By Lemma \ref{be}, it follows that the odd cohomology groups of $\Mb_{4,n}$ for $n \leq 3$ vanish and the even cohomology groups are pure Tate, with ranks given by corresponding coefficients of the polynomial. Using Arbarello and Cornalba's inductive method, they obtain the following theorem.

\begin{thm}[Bergstr\"om--Faber--Payne \cite{BFP}] \label{bfp}
Let $k = 7, 9$. Then $H^k(\Mb_{g,n}) = 0$ for all $g$ and $n$.
\end{thm}

\begin{rem}
The results of \cite{ckgp} give an alternative approach to 
proving the vanishing of odd cohomology for the base cases. Indeed, for every $g$ and $n$ with $d(g,n) \leq 9$, there is a filled circle in the chart in Section \ref{sec:taut}. Thus, by Corollary \ref{rhcor}, $H^*(\Mb_{g,n}) = RH^*(\Mb_{g,n})$, so all odd cohomology vanishes. Note however, that the genus $4$ cases lie beyond the range where computational techniques can determine the ranks of $RH^*(\Mb_{g,n})$, so this method does not determine the exact polynomial for the point counts, as was done in \cite{BFP} for $g = 4$ and $n \leq 3$.
\end{rem}

For $k = 11$, as we have seen earlier, $H^{11}(\Mb_{1,11}) \neq 0$, so there will be no uniform vanishing result. Nevertheless, for all $(g, n) \neq (1, 11)$ with $d(g,n) \leq 11$, there is a filled circle in the chart in Section \ref{sec:taut}, which means in particular that all odd cohomology vanishes. Thus, the only non-vanishing odd cohomology group in the base case region is $H^{11}(\Mb_{1,11})$. The following result is proved using improvements upon Arbarello and Cornalba's original induction. Below, for a partition $\lambda$ of $n$, $V_\lambda$ denotes the associated irreducible $\mathbb{S}_n$-representation.

\begin{thm}[Canning--L.--Payne \cite{h11}] \label{clp11}
We have $H^{11}(\Mb_{g,n}) = 0$ for all $g$ and $n$ unless $g = 1$ and $n \geq 11$ in which case $H^{11}(\Mb_{1,n})$ is generated by pullbacks along the forgetful maps $\Mb_{1,n} \to \Mb_{1,11}$. In fact, we have an $\mathbb{S}_n$-equivariant isomorphism of Galois representations
\[H^{11}(\Mb_{1,n}) = V_{n-10,1^{10}} \otimes H^{11}(\Mb_{1,11}) = V_{n-10,1^{10}} \otimes \mathsf{S}_{12}.\]
\end{thm} 

Continuing on to $k = 13$, there are several $(g, n)$ with $d(g, n) \leq 13$ that lie beyond the filled circles in Section \ref{sec:taut}. Nevertheless, for every $(g,n)$ with $d(g,n) \leq 13$ and no filled circle, we have either $g \leq 2$ or an open circle at $(g,n)$. For $g \leq 2$ work of Petersen shows that all degree $13$ cohomology is pushed forward from the boundary. For the cases with $g \geq 3$, the presence of an open circle together with Lemma \ref{open-ckgp}, also implies that all degree $13$ cohomology is pushed forward from the boundary.
Since $H^{11}$ is understood by Theorem \ref{clp11}, this gives adequate understanding of these base cases to run an inductive argument in the style of Arbarello and Cornalba's $k = 2$ case.
This succeeds in determining $H^{13}(\Mb_{g,n})$ with generators and relations. 

\begin{thm}[Canning--L.--Payne--Willwacher \cite{h13}] \label{clpw13}
The cohomology group $H^{13}(\Mb_{g,n})$ is spanned by the images of the pushforward maps
\[\iota_A: H^{11}(\Mb_{1,A \cup p}) \otimes H^0(\Mb_{g-1,A^c \cup p'}) \rightarrow H^{13}(\Mb_{g,n}).\]
If $g \geq 2$ then the sum $\bigoplus_{A} \iota_{A*}$ is also injective. In fact, for $g \geq 2$, we have an $\mathbb{S}_n$-equivariant isomorphism of Galois representations
\[H^{13}(\Mb_{g,n}) = \left(\bigoplus_{m=10}^n \mathrm{Ind}_{\mathbb{S}_m \times \mathbb{S}_{n-m}}^{\mathbb{S}_n} \left( (V_{m-9,1^9} \oplus V_{m-10,1^{10}}) \boxtimes \mathbf{1}\right)\right)\otimes \mathsf{LS}_{12}.\]
\end{thm}

Although similarly explicit presentations of $H^k(\Mb_{g,n})$ are not known for other values of $k$, 
using variations on Arbarello and Cornalba's induction together with the chart in Section \ref{sec:taut} to treat bases cases,
Conjecture \ref{lc} has been confirmed for $X = \Mb_{g,n}$ for $k \leq 15$.

\begin{thm}[Canning--L.--Payne \cite{h11,ste}] \label{lowwt}
For all $g$ and $n$, and for any $k \leq 15$, we have isomorphisms of Galois representations
\[H^{k}(\Mb_{g,n},\qq_\ell)^{\mathrm{ss}} = \begin{cases} \bigoplus \mathsf{L}^{k/2} & \text{if $k$ even} \\
\bigoplus \mathsf{S}_{12} & \text{if $k = 11$} \\
\bigoplus \mathsf{LS}_{12}  &\text{if $k = 13$} \\
\bigoplus \mathsf{L^2S}_{12} \oplus \bigoplus \mathsf{S}_{16} & \text{if $k = 15$} \\
0 & \text{otherwise}.
\end{cases}\]
\end{thm}

At present, we do not have sufficient information about base cases needed to push Theorem \ref{lowwt} beyond weight $15$. However, known rationality results for small moduli spaces  (which imply vanishing of the spaces of holomorphic forms) have enabled a description of
the Hodge component $H^{17,0}(\Mb_{g,n})$ and, conditional upon the conjectured vanishing of some spaces of holomorphic forms in genus $3$, a description of $H^{19,0}(\Mb_{g,n})$ as well, see \cite{FAmodules}.

\medskip
Conjecture \ref{lc} predicts that for even $k \leq 20$, all cohomology of $\Mb_{g,n}$ is pure Tate. By the Tate conjecture, such classes are expected to be algebraic. Meanwhile, all known examples of non-tautological algebraic cycles, such as those in Theorems \ref{vzx} and \ref{mainx},
occur in degrees $k \geq 22$. This suggests the following conjecutre.
\begin{conj} \label{lowk}
For even $k \leq 20$ and any $g, n$, we have $H^{k}(\Mb_{g,n}) = RH^k(\Mb_{g,n})$.
\end{conj}

The conjecture is known to hold for $k = 2$ \cite[Theorem 1.1]{ArbarelloCornalba}, for $k = 4$, and for $k = 6$ if $g \geq 10$ \cite[Theorem 1.5(1) and (2)]{ste}. The conjecture is also known when $g = 1$ (see Theorem \ref{g1}) and when $g = 2$ \cite[Corollary 2.7]{ste}.
By Hard Lefschetz, Conjecture \ref{lowk} would also imply that $H_k(\Mb_{g,n}) \colonequals H^{2(3g - 3 + n) - k}(\Mb_{g,n})$ is tautological for even $k \leq 20$. As further evidence towards Conjecture \ref{lowk}, \cite[Theorem 1.5(3)]{ste} shows that $H_k(\Mb_{g,n}) = RH_k(\Mb_{g,n})$ for even $k \leq 14$. In these cases where $H_k(\Mb_{g,n}) = RH_k(\Mb_{g,n})$, Poincar\'e duality shows that
Conjecture \ref{lowk} holds if and only the pairing $RH^k(\Mb_{g,n}) \times RH_k(\Mb_{g,n}) \to \qq$ is perfect.

\subsection{Low weight compactly supported cohomology of $\M_{g,n}$} \label{41}
Once $H^k(\Mb_{g,n})$ is understood for all $g$ and $n$, one might hope to use Deligne's weight spectral sequence to calculate $\gr_k H_c^*(\M_{g,n})$ as discussed in Section \ref{wt}. 
Since $H^0(\Mb_{g,n}) = \qq$ for all $g$ and $n$, the weight $0$ cohomology reflects purely combinatorial aspects of the boundary stratification.
This idea was pioneered in the case $k = 0$ by Chan--Galatius--Payne \cite{CGP,CGP2}, where they find that $\gr_0 H_c^*(\M_{g,n})= W_0H^*_c(\M_{g,n})$ is computed by  the cohomology of the genus $g$ part of the standard commutative graph complex with $n$ marked external vertices. Dualizing, they use this to prove at least exponential growth of the top weight cohomology of $\M_g$.

Building upon Theorems \ref{ac2} and \ref{clp11}, the weight $k$ compactly supported cohomology of $\M_{g,n}$ for
 $k = 2$ and $k = 11$ are studied by Payne and Willwacher in \cite{PW} and \cite{PW2} respectively. Additional progress for $k = 11$ was made in \cite{ThomasStudent}. Finally, building upon Theorem \ref{clpw13}, some results on the weight $k = 13$ cohomology are given in \cite[Proposition 5.1]{h13} and \cite{lstw}.

\section{Applications to point counts} \label{sec:pts}

\begin{figure}
\begin{center}
\begin{tikzpicture}[scale=.6]
\filldraw (0, 0) circle (4pt);
\node at (6,0) {$\#\Mb_{g,n}(\ff_q)$ is a polynomial in $q$};
\node[scale = .9, color = red] at (0, -1) {$\boldsymbol{\times}$};
\node at (6.5,-1) {$\#\Mb_{g,n}(\ff_q)$ is not a polynomial in $q$};
\end{tikzpicture}

\vspace{-.2in}
\begin{tikzpicture}[scale = .6]

\node[scale = .85] at (1, -.7) {$1$};
\node[scale = .85] at (2, -.7) {$2$};
\node[scale = .85] at (3, -.7) {$3$};
\node[scale = .85] at (4, -.7) {$4$};
\node[scale = .85] at (5, -.7) {$5$};
\node[scale = .85] at (6, -.7) {$6$};
\node[scale = .85] at (7, -.7) {$7$};
\node[scale = .85] at (8, -.7) {$8$};
\node[scale = .85] at (9, -.7) {$9$};
\node[scale = .85] at (10, -.7) {$10$};
\node[scale = .85] at (11, -.7) {$11$};
\node[scale = .85] at (12, -.7) {$12$};
\node[scale = .85] at (13, -.7) {$13$};
\node[scale = .85] at (14, -.7) {$14$};
\node[scale = .85] at (15, -.7) {$15$};
\node[scale = .85] at (16, -.7) {$16$};
\node[scale = .85] at (17, -.7) {$17$};

\node[scale = .85] at (-.7, 1) {$1$};
\node[scale = .85] at (-.7, 2) {$2$};
\node[scale = .85] at (-.7, 3) {$3$};
\node[scale = .85] at (-.7, 4) {$4$};
\node[scale = .85] at (-.7, 5) {$5$};
\node[scale = .85] at (-.7, 6) {$6$};
\node[scale = .85] at (-.7, 7) {$7$};
\node[scale = .85] at (-.7, 8) {$8$};
\node[scale = .85] at (-.7, 9) {$9$};
\node[scale = .85] at (-.7, 10) {$10$};
\node[scale = .85] at (-.7, 11) {$11$};
\node[scale = .85] at (-.7, 12) {$12$};
\node[scale = .85] at (-.7, 13) {$13$};

\draw[->] (0, 0) -- (18, 0);
\draw[->] (0, 0) -- (0, 14);
\node[scale=1.2] at (18.4,0) {$g$};
\node[scale=1.2] at (0, 14.4) {$n$};
\draw (-.1, 1) -- (.1, 1);
\draw (-.1, 2) -- (.1, 2);
\draw (1, -.1) -- (1, .1);
\draw (2, -.1) -- (2, .1);
\draw (3, -.1) -- (3, .1);
\filldraw (0, 3) circle (4pt);
\filldraw (0, 4) circle (4pt);
\filldraw (0, 5) circle (4pt);
\filldraw (0, 6) circle (4pt);
\filldraw (0, 7) circle (4pt);
\filldraw (0, 8) circle (4pt);
\filldraw (0, 9) circle (4pt);
\filldraw (0, 10) circle (4pt);
\filldraw (0, 11) circle (4pt);
\filldraw (0, 12) circle (4pt);
\filldraw (0, 13) circle (4pt);

\filldraw (1, 1) circle (4pt);
\filldraw (1, 2) circle (4pt);
\filldraw (1, 3) circle (4pt);
\filldraw (1, 4) circle (4pt);
\filldraw (1, 5) circle (4pt);
\filldraw (1, 6) circle (4pt);
\filldraw (1, 7) circle (4pt);
\filldraw (1, 8) circle (4pt);
\filldraw (1, 9) circle (4pt);
\filldraw (1, 10) circle (4pt);

\filldraw (2, 0) circle (4pt);
\filldraw (2, 1) circle (4pt);
\filldraw (3, 0) circle (4pt);

\node[scale = .9, color = red] at (1, 11) {$\boldsymbol{\times}$};
\node[scale = .9, color = red] at (1, 12) {$\boldsymbol{\times}$};
\node[scale = .9, color = red] at (1, 13) {$\boldsymbol{\times}$};

\node[scale = .9, color = red] at (2, 10) {$\boldsymbol{\times}$};
\node[scale = .9, color = red] at (2, 11) {$\boldsymbol{\times}$};
\node[scale = .9, color = red] at (2, 12) {$\boldsymbol{\times}$};
\node[scale = .9, color = red] at (2, 13) {$\boldsymbol{\times}$};

\node[scale = .9, color = red] at (3, 10) {$\boldsymbol{\times}$};
\node[scale = .9, color = red] at (3, 11) {$\boldsymbol{\times}$};
\node[scale = .9, color = red] at (3, 12) {$\boldsymbol{\times}$};
\node[scale = .9, color = red] at (3, 13) {$\boldsymbol{\times}$};

\node[scale = .9, color = red] at (4, 9) {$\boldsymbol{\times}$};
\node[scale = .9, color = red] at (4, 10) {$\boldsymbol{\times}$};
\node[scale = .9, color = red] at (4, 11) {$\boldsymbol{\times}$};
\node[scale = .9, color = red] at (4, 12) {$\boldsymbol{\times}$};
\node[scale = .9, color = red] at (4, 13) {$\boldsymbol{\times}$};

\node[scale = .9, color = red] at (5, 9) {$\boldsymbol{\times}$};
\node[scale = .9, color = red] at (5, 10) {$\boldsymbol{\times}$};
\node[scale = .9, color = red] at (5, 11) {$\boldsymbol{\times}$};
\node[scale = .9, color = red] at (5, 12) {$\boldsymbol{\times}$};
\node[scale = .9, color = red] at (5, 13) {$\boldsymbol{\times}$};

\node[scale = .9, color = red] at (6, 9) {$\boldsymbol{\times}$};
\node[scale = .9, color = red] at (6, 10) {$\boldsymbol{\times}$};
\node[scale = .9, color = red] at (6, 11) {$\boldsymbol{\times}$};
\node[scale = .9, color = red] at (6, 12) {$\boldsymbol{\times}$};
\node[scale = .9, color = red] at (6, 13) {$\boldsymbol{\times}$};

\node[scale = .9, color = red] at (7, 8) {$\boldsymbol{\times}$};
\node[scale = .9, color = red] at (7, 9) {$\boldsymbol{\times}$};
\node[scale = .9, color = red] at (7, 10) {$\boldsymbol{\times}$};
\node[scale = .9, color = red] at (7, 11) {$\boldsymbol{\times}$};
\node[scale = .9, color = red] at (7, 12) {$\boldsymbol{\times}$};
\node[scale = .9, color = red] at (7, 13) {$\boldsymbol{\times}$};

\node[scale =  .9, color = red] at (8, 8) {$\boldsymbol{\times}$};
\node[scale = .9, color = red] at (8, 9) {$\boldsymbol{\times}$};
\node[scale = .9, color = red] at (8, 10) {$\boldsymbol{\times}$};
\node[scale = .9, color = red] at (8, 11) {$\boldsymbol{\times}$};
\node[scale = .9, color = red] at (8, 12) {$\boldsymbol{\times}$};
\node[scale = .9, color = red] at (8, 13) {$\boldsymbol{\times}$};

\node[scale = .9, color = red] at (9, 8) {$\boldsymbol{\times}$};
\node[scale = .9, color = red] at (9, 9) {$\boldsymbol{\times}$};
\node[scale = .9, color = red] at (9, 10) {$\boldsymbol{\times}$};
\node[scale = .9, color = red] at (9, 11) {$\boldsymbol{\times}$};
\node[scale = .9, color = red] at (9, 12) {$\boldsymbol{\times}$};
\node[scale = .9, color = red] at (9, 13) {$\boldsymbol{\times}$};

\node[scale = .9, color = red] at (10, 8) {$\boldsymbol{\times}$};
\node[scale = .9, color = red] at (10, 9) {$\boldsymbol{\times}$};
\node[scale = .9, color = red] at (10, 10) {$\boldsymbol{\times}$};
\node[scale = .9, color = red] at (10, 11) {$\boldsymbol{\times}$};
\node[scale = .9, color = red] at (10, 12) {$\boldsymbol{\times}$};
\node[scale = .9, color = red] at (10, 13) {$\boldsymbol{\times}$};

\node[scale = .9, color = red] at (11, 7) {$\boldsymbol{\times}$};
\node[scale = .9, color = red] at (11, 8) {$\boldsymbol{\times}$};
\node[scale = .9, color = red] at (11, 9) {$\boldsymbol{\times}$};
\node[scale = .9, color = red] at (11, 10) {$\boldsymbol{\times}$};
\node[scale = .9, color = red] at (11, 11) {$\boldsymbol{\times}$};
\node[scale = .9, color = red] at (11, 12) {$\boldsymbol{\times}$};
\node[scale = .9, color = red] at (11, 13) {$\boldsymbol{\times}$};

\node[scale = .9, color = red] at (12, 7) {$\boldsymbol{\times}$};
\node[scale = .9, color = red] at (12, 8) {$\boldsymbol{\times}$};
\node[scale = .9, color = red] at (12, 9) {$\boldsymbol{\times}$};
\node[scale = .9, color = red] at (12, 10) {$\boldsymbol{\times}$};
\node[scale = .9, color = red] at (12, 11) {$\boldsymbol{\times}$};
\node[scale = .9, color = red] at (12, 12) {$\boldsymbol{\times}$};
\node[scale = .9, color = red] at (12, 13) {$\boldsymbol{\times}$};

\node[scale = .9, color = red] at (13, 7) {$\boldsymbol{\times}$};
\node[scale = .9, color = red] at (13, 8) {$\boldsymbol{\times}$};
\node[scale = .9, color = red] at (13, 9) {$\boldsymbol{\times}$};
\node[scale = .9, color = red] at (13, 10) {$\boldsymbol{\times}$};
\node[scale = .9, color = red] at (13, 11) {$\boldsymbol{\times}$};
\node[scale = .9, color = red] at (13, 12) {$\boldsymbol{\times}$};
\node[scale = .9, color = red] at (13, 13) {$\boldsymbol{\times}$};

\node[scale = .9, color = red] at (14, 7) {$\boldsymbol{\times}$};
\node[scale = .9, color = red] at (14, 8) {$\boldsymbol{\times}$};
\node[scale = .9, color = red] at (14, 9) {$\boldsymbol{\times}$};
\node[scale = .9, color = red] at (14, 10) {$\boldsymbol{\times}$};
\node[scale = .9, color = red] at (14, 11) {$\boldsymbol{\times}$};
\node[scale = .9, color = red] at (14, 12) {$\boldsymbol{\times}$};
\node[scale = .9, color = red] at (14, 13) {$\boldsymbol{\times}$};

\node[scale = .9, color = red] at (15, 7) {$\boldsymbol{\times}$};
\node[scale = .9, color = red] at (15, 8) {$\boldsymbol{\times}$};
\node[scale = .9, color = red] at (15, 9) {$\boldsymbol{\times}$};
\node[scale = .9, color = red] at (15, 10) {$\boldsymbol{\times}$};
\node[scale = .9, color = red] at (15, 11) {$\boldsymbol{\times}$};
\node[scale = .9, color = red] at (15, 12) {$\boldsymbol{\times}$};
\node[scale = .9, color = red] at (15, 13) {$\boldsymbol{\times}$};

\node[scale = .9, color = red] at (16, 0) {$\boldsymbol{\times}$};
\node[scale = .9, color = red] at (16, 1) {$\boldsymbol{\times}$};
\node[scale = .9, color = red] at (16, 2) {$\boldsymbol{\times}$};
\node[scale = .9, color = red] at (16, 3) {$\boldsymbol{\times}$};
\node[scale = .9, color = red] at (16, 4) {$\boldsymbol{\times}$};
\node[scale = .9, color = red] at (16, 5) {$\boldsymbol{\times}$};
\node[scale = .9, color = red] at (16, 6) {$\boldsymbol{\times}$};
\node[scale = .9, color = red] at (16, 7) {$\boldsymbol{\times}$};
\node[scale = .9, color = red] at (16, 8) {$\boldsymbol{\times}$};
\node[scale = .9, color = red] at (16, 9) {$\boldsymbol{\times}$};
\node[scale = .9, color = red] at (16, 10) {$\boldsymbol{\times}$};
\node[scale = .9, color = red] at (16, 11) {$\boldsymbol{\times}$};
\node[scale = .9, color = red] at (16, 12) {$\boldsymbol{\times}$};
\node[scale = .9, color = red] at (16, 13) {$\boldsymbol{\times}$};

\node[scale = .9, color = red] at (17, 0) {$\boldsymbol{\times}$};
\node[scale = .9, color = red] at (17, 1) {$\boldsymbol{\times}$};
\node[scale = .9, color = red] at (17, 2) {$\boldsymbol{\times}$};
\node[scale = .9, color = red] at (17, 3) {$\boldsymbol{\times}$};
\node[scale = .9, color = red] at (17, 4) {$\boldsymbol{\times}$};
\node[scale = .9, color = red] at (17, 5) {$\boldsymbol{\times}$};
\node[scale = .9, color = red] at (17, 6) {$\boldsymbol{\times}$};
\node[scale = .9, color = red] at (17, 7) {$\boldsymbol{\times}$};
\node[scale = .9, color = red] at (17, 8) {$\boldsymbol{\times}$};
\node[scale = .9, color = red] at (17, 9) {$\boldsymbol{\times}$};
\node[scale = .9, color = red] at (17, 10) {$\boldsymbol{\times}$};
\node[scale = .9, color = red] at (17, 11) {$\boldsymbol{\times}$};
\node[scale = .9, color = red] at (17, 12) {$\boldsymbol{\times}$};
\node[scale = .9, color = red] at (17, 13) {$\boldsymbol{\times}$};

\filldraw[color=black] (2,2) circle (4pt);
\filldraw[color=black] (2,3) circle (4pt);
\filldraw[color=black] (2,4) circle (4pt);
\filldraw[color=black] (2,5) circle (4pt);
\filldraw[color=black] (2,6) circle (4pt);
\filldraw[color=black] (2,7) circle (4pt);
\filldraw[color=black] (2,8) circle (4pt);
\filldraw[color=black] (2,9) circle (4pt);

\filldraw[color=black] (3,1) circle (4pt);
\filldraw[color=black] (3,2) circle (4pt);
\filldraw[color=black] (3,3) circle (4pt);
\filldraw[color=black] (3,4) circle (4pt);
\filldraw[color=black] (3,5) circle (4pt);
\filldraw[color=black] (3,6) circle (4pt);
\filldraw[color=black] (3,7) circle (4pt);
\filldraw[color=black] (3,8) circle (4pt);
\filldraw[color=black] (3,9) circle (4pt);

\filldraw[color=black] (4,0) circle (4pt);
\filldraw[color=black] (4,1) circle (4pt);
\filldraw[color=black] (4,2) circle (4pt);
\filldraw[color=black] (4,3) circle (4pt);
\filldraw[color=black] (4,4) circle (4pt);
\filldraw[color=black] (4,5) circle (4pt);
\filldraw[color=black] (4,6) circle (4pt);
\filldraw[color=black] (4,7) circle (4pt);

\filldraw[color=black] (5,0) circle (4pt);
\filldraw[color=black] (5,1) circle (4pt);
\filldraw[color=black] (5,2) circle (4pt);
\filldraw[color=black] (5,3) circle (4pt);
\filldraw[color=black] (5,4) circle (4pt);
\filldraw[color=black] (5,5) circle (4pt);

\filldraw[color=black] (6,0) circle (4pt);
\filldraw[color=black] (6,1) circle (4pt);
\filldraw[color=black] (6,2) circle (4pt);
\filldraw[color=black] (6,3) circle (4pt);

\filldraw[color=black] (7,0) circle (4pt);
\filldraw[color=black] (7,1) circle (4pt);

\draw[color=red, thick, ->] (18,14) -- (19,15);
\end{tikzpicture}
\end{center}

\vspace{.3in}

\begin{center}
\begin{tikzpicture}[scale=.6]
\draw[ultra thick] (0, 0) circle (4pt);
\node at (6,0) {$\#\M_{g,n}(\ff_q)$ is a polynomial in $q$};
\node[scale = .9, color = red] at (0, -1) {$\boldsymbol{\times}$};
\node at (6.5,-1) {$\#\Mb_{g,n}(\ff_q)$ is not a polynomial in $q$};
\end{tikzpicture}

\vspace{-.2in}
\begin{tikzpicture}[scale = .6]

\node[scale = .85] at (1, -.7) {$1$};
\node[scale = .85] at (2, -.7) {$2$};
\node[scale = .85] at (3, -.7) {$3$};
\node[scale = .85] at (4, -.7) {$4$};
\node[scale = .85] at (5, -.7) {$5$};
\node[scale = .85] at (6, -.7) {$6$};
\node[scale = .85] at (7, -.7) {$7$};
\node[scale = .85] at (8, -.7) {$8$};
\node[scale = .85] at (9, -.7) {$9$};
\node[scale = .85] at (10, -.7) {$10$};
\node[scale = .85] at (11, -.7) {$11$};
\node[scale = .85] at (12, -.7) {$12$};
\node[scale = .85] at (13, -.7) {$13$};
\node[scale = .85] at (14, -.7) {$14$};
\node[scale = .85] at (15, -.7) {$15$};
\node[scale = .85] at (16, -.7) {$16$};
\node[scale = .85] at (17, -.7) {$17$};

\node[scale = .85] at (-.7, 1) {$1$};
\node[scale = .85] at (-.7, 2) {$2$};
\node[scale = .85] at (-.7, 3) {$3$};
\node[scale = .85] at (-.7, 4) {$4$};
\node[scale = .85] at (-.7, 5) {$5$};
\node[scale = .85] at (-.7, 6) {$6$};
\node[scale = .85] at (-.7, 7) {$7$};
\node[scale = .85] at (-.7, 8) {$8$};
\node[scale = .85] at (-.7, 9) {$9$};
\node[scale = .85] at (-.7, 10) {$10$};
\node[scale = .85] at (-.7, 11) {$11$};
\node[scale = .85] at (-.7, 12) {$12$};
\node[scale = .85] at (-.7, 13) {$13$};

\draw[->] (0, 0) -- (18, 0);
\draw[->] (0, 0) -- (0, 14);
\node[scale=1.2] at (18.4,0) {$g$};
\node[scale=1.2] at (0, 14.4) {$n$};
\draw (-.1, 1) -- (.1, 1);
\draw (-.1, 2) -- (.1, 2);
\draw (1, -.1) -- (1, .1);
\draw (2, -.1) -- (2, .1);
\draw (3, -.1) -- (3, .1);

\filldraw[color=white] (0, 3) circle (4pt);
\filldraw[color=white] (0, 4) circle (4pt);
\filldraw[color=white] (0, 5) circle (4pt);
\filldraw[color=white] (0, 6) circle (4pt);
\filldraw[color=white] (0, 7) circle (4pt);
\filldraw[color=white] (0, 8) circle (4pt);
\filldraw[color=white] (0, 9) circle (4pt);
\filldraw[color=white] (0, 10) circle (4pt);
\filldraw[color=white] (0, 11) circle (4pt);
\filldraw[color=white] (0, 12) circle (4pt);
\filldraw[color=white] (0, 13) circle (4pt);

\draw[ultra thick] (0, 3) circle (4pt);
\draw[ultra thick] (0, 4) circle (4pt);
\draw[ultra thick] (0, 5) circle (4pt);
\draw[ultra thick] (0, 6) circle (4pt);
\draw[ultra thick] (0, 7) circle (4pt);
\draw[ultra thick] (0, 8) circle (4pt);
\draw[ultra thick] (0, 9) circle (4pt);
\draw[ultra thick] (0, 10) circle (4pt);
\draw[ultra thick] (0, 11) circle (4pt);
\draw[ultra thick] (0, 12) circle (4pt);
\draw[ultra thick] (0, 13) circle (4pt);

\draw[ultra thick] (1, 1) circle (4pt);
\draw[ultra thick] (1, 2) circle (4pt);
\draw[ultra thick] (1, 3) circle (4pt);
\draw[ultra thick] (1, 4) circle (4pt);
\draw[ultra thick] (1, 5) circle (4pt);
\draw[ultra thick] (1, 6) circle (4pt);
\draw[ultra thick] (1, 7) circle (4pt);
\draw[ultra thick] (1, 8) circle (4pt);
\draw[ultra thick] (1, 9) circle (4pt);
\draw[ultra thick] (1, 10) circle (4pt);

\node[scale = .9, color = red] at (1, 11) {$\boldsymbol{\times}$};
\node[scale = .9, color = red] at (1, 12) {$\boldsymbol{\times}$};
\node[scale = .9, color = red] at (1, 13) {$\boldsymbol{\times}$};

\node[scale = .9, color = red] at (2, 10) {$\boldsymbol{\times}$};
\node[scale = .9, color = red] at (2, 11) {$\boldsymbol{\times}$};
\node[scale = .9, color = red] at (2, 12) {$\boldsymbol{\times}$};
\node[scale = .9, color = red] at (2, 13) {$\boldsymbol{\times}$};

\node[scale = .9, color = red] at (3, 8) {$\boldsymbol{\times}$};
\node[scale = .9, color = red] at (3, 9) {$\boldsymbol{\times}$};
\node[scale = .9, color = red] at (3, 10) {$\boldsymbol{\times}$};
\node[scale = .9, color = red] at (3, 11) {$\boldsymbol{\times}$};
\node[scale = .9, color = red] at (3, 12) {$\boldsymbol{\times}$};
\node[scale = .9, color = red] at (3, 13) {$\boldsymbol{\times}$};

\node[scale = .9, color = red] at (4, 7) {$\boldsymbol{\times}$};
\node[scale = .9, color = red] at (4, 8) {$\boldsymbol{\times}$};
\node[scale = .9, color = red] at (4, 9) {$\boldsymbol{\times}$};
\node[scale = .9, color = red] at (4, 10) {$\boldsymbol{\times}$};
\node[scale = .9, color = red] at (4, 11) {$\boldsymbol{\times}$};
\node[scale = .9, color = red] at (4, 12) {$\boldsymbol{\times}$};
\node[scale = .9, color = red] at (4, 13) {$\boldsymbol{\times}$};

\node[scale = .9, color = red] at (5, 5) {$\boldsymbol{\times}$};
\node[scale = .9, color = red] at (5, 6) {$\boldsymbol{\times}$};
\node[scale = .9, color = red] at (5, 7) {$\boldsymbol{\times}$};
\node[scale = .9, color = red] at (5, 8) {$\boldsymbol{\times}$};
\node[scale = .9, color = red] at (5, 9) {$\boldsymbol{\times}$};
\node[scale = .9, color = red] at (5, 10) {$\boldsymbol{\times}$};
\node[scale = .9, color = red] at (5, 11) {$\boldsymbol{\times}$};
\node[scale = .9, color = red] at (5, 12) {$\boldsymbol{\times}$};
\node[scale = .9, color = red] at (5, 13) {$\boldsymbol{\times}$};

\node[scale = .9, color = red] at (6, 4) {$\boldsymbol{\times}$};
\node[scale = .9, color = red] at (6, 5) {$\boldsymbol{\times}$};
\node[scale = .9, color = red] at (6, 6) {$\boldsymbol{\times}$};
\node[scale = .9, color = red] at (6, 7) {$\boldsymbol{\times}$};
\node[scale = .9, color = red] at (6, 8) {$\boldsymbol{\times}$};
\node[scale = .9, color = red] at (6, 9) {$\boldsymbol{\times}$};
\node[scale = .9, color = red] at (6, 10) {$\boldsymbol{\times}$};
\node[scale = .9, color = red] at (6, 11) {$\boldsymbol{\times}$};
\node[scale = .9, color = red] at (6, 12) {$\boldsymbol{\times}$};
\node[scale = .9, color = red] at (6, 13) {$\boldsymbol{\times}$};

\node[scale = .9, color = red] at (7, 2) {$\boldsymbol{\times}$};
\node[scale = .9, color = red] at (7, 3) {$\boldsymbol{\times}$};
\node[scale = .9, color = red] at (7, 4) {$\boldsymbol{\times}$};
\node[scale = .9, color = red] at (7, 5) {$\boldsymbol{\times}$};
\node[scale = .9, color = red] at (7, 6) {$\boldsymbol{\times}$};
\node[scale = .9, color = red] at (7, 7) {$\boldsymbol{\times}$};
\node[scale = .9, color = red] at (7, 8) {$\boldsymbol{\times}$};
\node[scale = .9, color = red] at (7, 9) {$\boldsymbol{\times}$};
\node[scale = .9, color = red] at (7, 10) {$\boldsymbol{\times}$};
\node[scale = .9, color = red] at (7, 11) {$\boldsymbol{\times}$};
\node[scale = .9, color = red] at (7, 12) {$\boldsymbol{\times}$};
\node[scale = .9, color = red] at (7, 13) {$\boldsymbol{\times}$};

\node[scale = .9, color = red] at (8, 1) {$\boldsymbol{\times}$};
\node[scale = .9, color = red] at (8, 2) {$\boldsymbol{\times}$};
\node[scale = .9, color = red] at (8, 3) {$\boldsymbol{\times}$};
\node[scale = .9, color = red] at (8, 4) {$\boldsymbol{\times}$};
\node[scale = .9, color = red] at (8, 5) {$\boldsymbol{\times}$};
\node[scale = .9, color = red] at (8, 6) {$\boldsymbol{\times}$};
\node[scale = .9, color = red] at (8, 7) {$\boldsymbol{\times}$};
\node[scale =  .9, color = red] at (8, 8) {$\boldsymbol{\times}$};
\node[scale = .9, color = red] at (8, 9) {$\boldsymbol{\times}$};
\node[scale = .9, color = red] at (8, 10) {$\boldsymbol{\times}$};
\node[scale = .9, color = red] at (8, 11) {$\boldsymbol{\times}$};
\node[scale = .9, color = red] at (8, 12) {$\boldsymbol{\times}$};
\node[scale = .9, color = red] at (8, 13) {$\boldsymbol{\times}$};

\node[scale = .9, color = red] at (9, 0) {$\boldsymbol{\times}$};
\node[scale = .9, color = red] at (9, 1) {$\boldsymbol{\times}$};
\node[scale = .9, color = red] at (9, 2) {$\boldsymbol{\times}$};
\node[scale = .9, color = red] at (9, 3) {$\boldsymbol{\times}$};
\node[scale = .9, color = red] at (9, 4) {$\boldsymbol{\times}$};
\node[scale = .9, color = red] at (9, 5) {$\boldsymbol{\times}$};
\node[scale = .9, color = red] at (9, 6) {$\boldsymbol{\times}$};
\node[scale = .9, color = red] at (9, 7) {$\boldsymbol{\times}$};
\node[scale = .9, color = red] at (9, 8) {$\boldsymbol{\times}$};
\node[scale = .9, color = red] at (9, 9) {$\boldsymbol{\times}$};
\node[scale = .9, color = red] at (9, 10) {$\boldsymbol{\times}$};
\node[scale = .9, color = red] at (9, 11) {$\boldsymbol{\times}$};
\node[scale = .9, color = red] at (9, 12) {$\boldsymbol{\times}$};
\node[scale = .9, color = red] at (9, 13) {$\boldsymbol{\times}$};

\node[scale = .9, color = red] at (10, 0) {$\boldsymbol{\times}$};
\node[scale = .9, color = red] at (10, 1) {$\boldsymbol{\times}$};
\node[scale = .9, color = red] at (10, 2) {$\boldsymbol{\times}$};
\node[scale = .9, color = red] at (10, 3) {$\boldsymbol{\times}$};
\node[scale = .9, color = red] at (10, 4) {$\boldsymbol{\times}$};
\node[scale = .9, color = red] at (10, 5) {$\boldsymbol{\times}$};
\node[scale = .9, color = red] at (10, 6) {$\boldsymbol{\times}$};
\node[scale = .9, color = red] at (10, 7) {$\boldsymbol{\times}$};
\node[scale = .9, color = red] at (10, 8) {$\boldsymbol{\times}$};
\node[scale = .9, color = red] at (10, 9) {$\boldsymbol{\times}$};
\node[scale = .9, color = red] at (10, 10) {$\boldsymbol{\times}$};
\node[scale = .9, color = red] at (10, 11) {$\boldsymbol{\times}$};
\node[scale = .9, color = red] at (10, 12) {$\boldsymbol{\times}$};
\node[scale = .9, color = red] at (10, 13) {$\boldsymbol{\times}$};

\node[scale = .9, color = red] at (11, 0) {$\boldsymbol{\times}$};
\node[scale = .9, color = red] at (11, 1) {$\boldsymbol{\times}$};
\node[scale = .9, color = red] at (11, 2) {$\boldsymbol{\times}$};
\node[scale = .9, color = red] at (11, 3) {$\boldsymbol{\times}$};
\node[scale = .9, color = red] at (11, 4) {$\boldsymbol{\times}$};
\node[scale = .9, color = red] at (11, 5) {$\boldsymbol{\times}$};
\node[scale = .9, color = red] at (11, 6) {$\boldsymbol{\times}$};
\node[scale = .9, color = red] at (11, 7) {$\boldsymbol{\times}$};
\node[scale = .9, color = red] at (11, 8) {$\boldsymbol{\times}$};
\node[scale = .9, color = red] at (11, 9) {$\boldsymbol{\times}$};
\node[scale = .9, color = red] at (11, 10) {$\boldsymbol{\times}$};
\node[scale = .9, color = red] at (11, 11) {$\boldsymbol{\times}$};
\node[scale = .9, color = red] at (11, 12) {$\boldsymbol{\times}$};
\node[scale = .9, color = red] at (11, 13) {$\boldsymbol{\times}$};

\node[scale = .9, color = red] at (12, 0) {$\boldsymbol{\times}$};
\node[scale = .9, color = red] at (12, 1) {$\boldsymbol{\times}$};
\node[scale = .9, color = red] at (12, 2) {$\boldsymbol{\times}$};
\node[scale = .9, color = red] at (12, 3) {$\boldsymbol{\times}$};
\node[scale = .9, color = red] at (12, 4) {$\boldsymbol{\times}$};
\node[scale = .9, color = red] at (12, 5) {$\boldsymbol{\times}$};
\node[scale = .9, color = red] at (12, 6) {$\boldsymbol{\times}$};
\node[scale = .9, color = red] at (12, 7) {$\boldsymbol{\times}$};
\node[scale = .9, color = red] at (12, 8) {$\boldsymbol{\times}$};
\node[scale = .9, color = red] at (12, 9) {$\boldsymbol{\times}$};
\node[scale = .9, color = red] at (12, 10) {$\boldsymbol{\times}$};
\node[scale = .9, color = red] at (12, 11) {$\boldsymbol{\times}$};
\node[scale = .9, color = red] at (12, 12) {$\boldsymbol{\times}$};
\node[scale = .9, color = red] at (12, 13) {$\boldsymbol{\times}$};

\node[scale = .9, color = red] at (13, 0) {$\boldsymbol{\times}$};
\node[scale = .9, color = red] at (13, 1) {$\boldsymbol{\times}$};
\node[scale = .9, color = red] at (13, 2) {$\boldsymbol{\times}$};
\node[scale = .9, color = red] at (13, 3) {$\boldsymbol{\times}$};
\node[scale = .9, color = red] at (13, 4) {$\boldsymbol{\times}$};
\node[scale = .9, color = red] at (13, 5) {$\boldsymbol{\times}$};
\node[scale = .9, color = red] at (13, 6) {$\boldsymbol{\times}$};
\node[scale = .9, color = red] at (13, 7) {$\boldsymbol{\times}$};
\node[scale = .9, color = red] at (13, 8) {$\boldsymbol{\times}$};
\node[scale = .9, color = red] at (13, 9) {$\boldsymbol{\times}$};
\node[scale = .9, color = red] at (13, 10) {$\boldsymbol{\times}$};
\node[scale = .9, color = red] at (13, 11) {$\boldsymbol{\times}$};
\node[scale = .9, color = red] at (13, 12) {$\boldsymbol{\times}$};
\node[scale = .9, color = red] at (13, 13) {$\boldsymbol{\times}$};

\node[scale = .9, color = red] at (14, 0) {$\boldsymbol{\times}$};
\node[scale = .9, color = red] at (14, 1) {$\boldsymbol{\times}$};
\node[scale = .9, color = red] at (14, 2) {$\boldsymbol{\times}$};
\node[scale = .9, color = red] at (14, 3) {$\boldsymbol{\times}$};
\node[scale = .9, color = red] at (14, 4) {$\boldsymbol{\times}$};
\node[scale = .9, color = red] at (14, 5) {$\boldsymbol{\times}$};
\node[scale = .9, color = red] at (14, 6) {$\boldsymbol{\times}$};
\node[scale = .9, color = red] at (14, 7) {$\boldsymbol{\times}$};
\node[scale = .9, color = red] at (14, 8) {$\boldsymbol{\times}$};
\node[scale = .9, color = red] at (14, 9) {$\boldsymbol{\times}$};
\node[scale = .9, color = red] at (14, 10) {$\boldsymbol{\times}$};
\node[scale = .9, color = red] at (14, 11) {$\boldsymbol{\times}$};
\node[scale = .9, color = red] at (14, 12) {$\boldsymbol{\times}$};
\node[scale = .9, color = red] at (14, 13) {$\boldsymbol{\times}$};

\node[scale = .9, color = red] at (15, 0) {$\boldsymbol{\times}$};
\node[scale = .9, color = red] at (15, 1) {$\boldsymbol{\times}$};
\node[scale = .9, color = red] at (15, 2) {$\boldsymbol{\times}$};
\node[scale = .9, color = red] at (15, 3) {$\boldsymbol{\times}$};
\node[scale = .9, color = red] at (15, 4) {$\boldsymbol{\times}$};
\node[scale = .9, color = red] at (15, 5) {$\boldsymbol{\times}$};
\node[scale = .9, color = red] at (15, 6) {$\boldsymbol{\times}$};
\node[scale = .9, color = red] at (15, 7) {$\boldsymbol{\times}$};
\node[scale = .9, color = red] at (15, 8) {$\boldsymbol{\times}$};
\node[scale = .9, color = red] at (15, 9) {$\boldsymbol{\times}$};
\node[scale = .9, color = red] at (15, 10) {$\boldsymbol{\times}$};
\node[scale = .9, color = red] at (15, 11) {$\boldsymbol{\times}$};
\node[scale = .9, color = red] at (15, 12) {$\boldsymbol{\times}$};
\node[scale = .9, color = red] at (15, 13) {$\boldsymbol{\times}$};

\node[scale = .9, color = red] at (16, 0) {$\boldsymbol{\times}$};
\node[scale = .9, color = red] at (16, 1) {$\boldsymbol{\times}$};
\node[scale = .9, color = red] at (16, 2) {$\boldsymbol{\times}$};
\node[scale = .9, color = red] at (16, 3) {$\boldsymbol{\times}$};
\node[scale = .9, color = red] at (16, 4) {$\boldsymbol{\times}$};
\node[scale = .9, color = red] at (16, 5) {$\boldsymbol{\times}$};
\node[scale = .9, color = red] at (16, 6) {$\boldsymbol{\times}$};
\node[scale = .9, color = red] at (16, 7) {$\boldsymbol{\times}$};
\node[scale = .9, color = red] at (16, 8) {$\boldsymbol{\times}$};
\node[scale = .9, color = red] at (16, 9) {$\boldsymbol{\times}$};
\node[scale = .9, color = red] at (16, 10) {$\boldsymbol{\times}$};
\node[scale = .9, color = red] at (16, 11) {$\boldsymbol{\times}$};
\node[scale = .9, color = red] at (16, 12) {$\boldsymbol{\times}$};
\node[scale = .9, color = red] at (16, 13) {$\boldsymbol{\times}$};

\node[scale = .9, color = red] at (17, 0) {$\boldsymbol{\times}$};
\node[scale = .9, color = red] at (17, 1) {$\boldsymbol{\times}$};
\node[scale = .9, color = red] at (17, 2) {$\boldsymbol{\times}$};
\node[scale = .9, color = red] at (17, 3) {$\boldsymbol{\times}$};
\node[scale = .9, color = red] at (17, 4) {$\boldsymbol{\times}$};
\node[scale = .9, color = red] at (17, 5) {$\boldsymbol{\times}$};
\node[scale = .9, color = red] at (17, 6) {$\boldsymbol{\times}$};
\node[scale = .9, color = red] at (17, 7) {$\boldsymbol{\times}$};
\node[scale = .9, color = red] at (17, 8) {$\boldsymbol{\times}$};
\node[scale = .9, color = red] at (17, 9) {$\boldsymbol{\times}$};
\node[scale = .9, color = red] at (17, 10) {$\boldsymbol{\times}$};
\node[scale = .9, color = red] at (17, 11) {$\boldsymbol{\times}$};
\node[scale = .9, color = red] at (17, 12) {$\boldsymbol{\times}$};
\node[scale = .9, color = red] at (17, 13) {$\boldsymbol{\times}$};

\filldraw[color=white] (2,0) circle (4pt);
\draw[ultra thick] (2, 0) circle (4pt);
\draw[ultra thick] (2, 1) circle (4pt);
\draw[ultra thick] (2,2) circle (4pt);
\draw[ultra thick] (2,3) circle (4pt);
\draw[ultra thick] (2,4) circle (4pt);
\draw[ultra thick] (2,5) circle (4pt);
\draw[ultra thick] (2,6) circle (4pt);
\draw[ultra thick] (2,7) circle (4pt);
\draw[ultra thick] (2,8) circle (4pt);
\draw[ultra thick] (2,9) circle (4pt);

\filldraw[color=white] (3,0) circle (4pt);
\draw[ultra thick] (3, 0) circle (4pt);
\draw[ultra thick] (3, 1) circle (4pt);
\draw[ultra thick] (3,2) circle (4pt);
\draw[ultra thick] (3,3) circle (4pt);
\draw[ultra thick] (3,4) circle (4pt);
\draw[ultra thick] (3,5) circle (4pt);
\draw[ultra thick] (3,6) circle (4pt);
\draw[ultra thick] (3,7) circle (4pt);

\filldraw[color=white] (4,0) circle (4pt);
\draw[ultra thick] (4, 0) circle (4pt);
\draw[ultra thick] (4, 1) circle (4pt);
\draw[ultra thick] (4,2) circle (4pt);
\draw[ultra thick] (4,3) circle (4pt);
\draw[ultra thick] (4,4) circle (4pt);
\draw[ultra thick] (4,5) circle (4pt);
\draw[ultra thick] (4,6) circle (4pt);

\filldraw[color=white] (5,0) circle (4pt);
\draw[ultra thick] (5, 0) circle (4pt);
\draw[ultra thick] (5, 1) circle (4pt);
\draw[ultra thick] (5,2) circle (4pt);
\draw[ultra thick] (5,3) circle (4pt);
\draw[ultra thick] (5,4) circle (4pt);

\filldraw[color=white] (6,0) circle (4pt);
\draw[ultra thick] (6, 0) circle (4pt);
\draw[ultra thick] (6, 1) circle (4pt);
\draw[ultra thick] (6,2) circle (4pt);
\draw[ultra thick] (6,3) circle (4pt);

\filldraw[color=white] (7,0) circle (4pt);
\draw[ultra thick] (7, 0) circle (4pt);
\draw[ultra thick] (7, 1) circle (4pt);

\filldraw[color=white] (8,0) circle (4pt);
\draw[ultra thick] (8, 0) circle (4pt);
\node[color = red, scale = 1.5] at (18.5, 14.5) {?};
\draw[->, color=red, thick] (18,0) -- (19,0);
\end{tikzpicture}
\end{center}

\end{figure}

We now turn to the question of when $\#\Mb_{g,n}(\ff_q)$ or $\#\M_{g,n}(\ff_q)$ is given by a polynomial in $q$. In both cases, the tool for showing the point count is polynomial is to show all cohomology is Tate type and apply Lemma \ref{tate}.
Meanwhile, the tool for showing that the point count is not polynomial is to find some odd $k$ for which $\chi_k(\M) \neq 0$ and apply Lemma \ref{chik}. 
Because of the vanishing results for odd $k$ in Theorems \ref{ac2} and \ref{bfp}, to obtain non-vanishing in odd weight we will of course need to take $k > 9$.

For $\M = \Mb_{g,n}$, polynomial point counts are implied by $H^*(\Mb_{g,n}) = RH^*(\Mb_{g,n})$. We represent this with filled circles in the chart below. The positions in the chart below for $\#\Mb_{g,n}(\ff_q)$ that receive filled circles are 
all positions that are filled circles in the chart in Section \ref{sec:taut} together with the pairs listed in Proposition \ref{samir8}.
Meanwhile, since $\Mb_{g,n}$ is smooth and proper, $\chi_k(\Mb_{g,n}) = (-1) \dim H^k(\Mb_{g,n})$.
The non-vanshing odd cohomology classes constructed in  Theorems \ref{odd1} and \ref{odd2} give rise to the red exes
in the first chart for $\#\Mb_{g,n}(\ff_q)$, representing cases where the point count is not polynomial. All positions off the page are also filled with red exes.

For $\M = \M_{g,n}$, it is shown in \cite[Theorem 1.5]{h13} that if $3g + 2n < 25$, then the cohomology of $\M_{g,n}$ is Tate type, and hence these spaces have polynomial point count. These are represented by open circles in the second chart.
Since $\M_{g,n}$ is not proper,
 the proof of this result requires some care. For example, although we know that $\Mb_{2,9}$ has polynomial point count, it does not follow immediately that $\M_{2,9}$ has polynomial point count, as the boundary could have some non-polynomial contribution. One may be concerned that $\Mb_{1,11}$ maps to the boundary of $\Mb_{2,9}$ and $\Mb_{1,11}$ has non-polynomial point count. However, as discussed in Section \ref{equiv}, what matters for such calculations are the \emph{equivariant} point counts of $\Mb_{1,11}$.  Indeed, 
 the gluing map $\Mb_{1,11} \to \Mb_{2,9}$ factors through the $\mathbb{S}_2$-quotient of $\Mb_{1,11}$ that identifies the two possible orders of the points being glued together.
  In this case, even though $H^{11}(\Mb_{1,11}) \neq 0$, it turns out that $H^{11}(\Mb_{1,11})^{\mathbb{S}_{2}} = 0$. This is the reason behind why $\M_{2,9}$ still has polynomial point count.
 In summary, to prove that the cohomology of $\M_{g,n}$ for $3g + 2n < 25$ is Tate type, it is essential to understand \emph{equivariantly} the non-Tate cohomology of the $\Mb_{g',n'}$ with $2g' + n' \leq 2g + n$. 
Another other important input is that $\M_{g,n}$ has the CKgP, which implies by Lemma \ref{open-ckgp} that these non-Tate parts of the cohomology are pushed forward from the boundary, allowing us to understand them explicitly in terms of decorated graphs (with possibly non-tautological decorations).
  
It is expected that the point counts of $\M_{g,n}$ are non-polynomial outside of this range.
\begin{conj} \label{ppc}
$\M_{g,n}$ has polynomial point count if and only if $3g + 2n < 25$.
\end{conj}

In \cite{h13}, Conjecture \ref{ppc} is proved in the case $n = 0$ and computationally verified for all $g + n < 150$. We place a red {\color{red} ?} in the upper right of the chart for $\#\M_{g,n}(\ff_q)$ to indicate that it remains unknown if all spaces off of the page should be given red exes, as predicted by Conjecture \ref{ppc}.

To prove the conjecture for $n = 0$, we study the asymptotic behavior of 
a generating function for $\chi_{11}(\M_{g,n})$ given in \cite{PW2}. In particular, we bound it away from zero for $g > 500$ and compute the rest on a computer to prove that
\[\chi_{11}(\M_{g,0}) \neq 0 \text{ if and only if $g \geq 9$ and $g \neq 12$}.\]
This proves all red exes besides $(g, n) = (12, 0)$ on the $n = 0$ axis of the chart for $\#\M_{g,n}(\ff_q)$. For $(g, n) = (12, 0)$, we have $\chi_{11}(\M_{12,0}) = 0$!
Nevertheless, building upon Theorem \ref{clpw13}, we prove that
 $\chi_{13}(\M_{12,0}) \neq 0$.
Similarly, it turns out that $\chi_{11}(\M_{8,1}) = 0$, but $\chi_{13}(\M_{8,1}) \neq 0$. All other know red exes with $g + n < 150$ are proved by showing $\chi_{11}(\M_{g,n}) \neq 0$ in a computer implementation.

\subsection{Approximate polynomiality}
As a consequence of the chart for $\#\Mb_{g,n}(\ff_q)$, we see that  there are only finitely many cases where $g > 0$ and $\#\Mb_{g,n}(\ff_q)$ is a polynomial in $q$. Nevertheless, we may ask about the overall shape and leading terms in the point count. 
This is determined by the Galois representations that appear in $H^k(\Mb_{g,n})$ for high values of $k$ (which differ from those in complementary low degrees by a Tate twist).
As a corollary of Theorem \ref{lowk}, the point counts of $\Mb_{g,n}$ are surprisingly well-approximated by a polynomial. The first corrections to polynomial behavior are in weight $2\dim \Mb_{g,n} - 11$, but by Theorem \ref{clp11}, only show up when $g = 1$. That is for $g \geq 2$, we have \cite[Corollary 1.3]{h11}
\[\#\Mb_{g,n}(\ff_q) = P(q) + O(q^{\dim \Mb_{g,n} - 13/2}).\]

\subsection{Further point counts in genus $2$ and $3$}
Going beyond the shape of the point counts, Bergstr\"om and Faber have determined a compiled a list of the exact point counts of $\Mb_{g,n}$ and $\M_{g,n}$ for $g  \leq 3$ and many small values of $n$, see \cite{github} and the references therein for an overview of the known results.

To illustrate one of the strategies, used in \cite{BF} in genus $3$, first suppose we knew that $\#\Mb_{g,n}(\ff_q)$ were given by a polynomial in $q$, say
\[\#\Mb_{g,n}(\ff_q) = P(q) = a_d q^d + \ldots + a_1 q + a_0,\]
but we did not know the coefficients $a_i$. 
We could determine the coefficients $a_i$ --- and thus know $\#\Mb_{g,n}(\ff_q)$ for all $q$ --- just
 by knowing $\#\Mb_{g,n}(\ff_q)$ for finitely many values of $q$, as we now explain. Indeed, for each value of $q$ where the point count is known, plugging it in to the equation above yields a linear relation on the $a_i$. With enough independent linear relations among the $a_i$, one can solve the linear system.
There are also some other helpful relations.
Namely, by Lemma \ref{be} and Poincar\'e duality, we have $a_i = a_{d-i}$. Moreover, $\sum_{i} a_i = \chi(\Mb_{g,n})$, and this integer-valued Euler characteristic is computable for small $g$ and $n$, giving another relation. 

In theory, for small $q$ and $g \leq 3$, one can explicitly compute $\#\Mb_{g,n}(\ff_q)$, relying on explicit models of the curves in low genus.
In order to work with contributions from the boundary, as in Section \ref{equiv}, one should actually work with equivariant point counts, which amounts to counting points on twisted forms as well.
 Inducting on the genus and number of markings, one can assume the contribution from the boundary is understood. The contribution from the interior $\M_{g,n}$ (and each twisted form) can be made very explicit in low genus.
For example, in genus $3$, every curve is either hyperelliptic or a smooth plane quartic. For each small $q$, there are only finitely many smooth plane quartics which can be enumerated and studied to obtain these point counts.

Even if we do not know the point count is a polynomial, a similar strategy works if $\dim \Mb_{g,n} \leq 22$ and Conjecture \ref{lc} holds. In this case, the Grothendieck--Lefschetz trace formula implies that $\#\Mb_{g,n}(\ff_q)$ has the shape
\[\#\Mb_{g,n}(\ff_q) = \sum_{\mathsf{V}} a_{\mathsf{V}} \tr(\mathrm{Frob}_q^*| \mathsf{V}),\]
where the sum runs over the Galois representations $\mathsf{V}$ allowed by Conjecture \ref{lc}. Here,
the integers $a_{\mathsf{V}}$ are given by the alternating sum of the multiplicities of $\mathsf{V}$ in the cohomology groups.
For the representations $\mathsf{V}$ allowed by Conjecture \ref{lc}, the trace of Frobenius is determined by coefficients of corresponding modular forms. Thus, in a similar fashion to the previous paragraph, each time $\#\Mb_{g,n}(\ff_q) $ is known for some small value of $q$, it gives rise to a linear relation among the coefficients $a_{\mathsf{V}}$. This strategy is carried out in \cite{BF} to determine the $a_{\mathsf{V}}$ for $g = 3$ and $n \leq 14$. In fact, the equivariant versions of $a_{\mathsf{V}}$ are determined. In the answers presented at the github link above,
the coefficients in the equivariant point count $\#^{\mathbb{S}_n}\Mb_{g,n}$ are sums of irreducible $\mathbb{S}_n$ representations; taking dimensions of these representations would give the coefficients $a_{\mathsf{V}}$.

\begin{rem}
The hypothesis $n \leq 14$ is used in \cite{BF} in order to know that $\Mb_{3,n}$ is unirational. This forces the relations $a_{\mathsf{V}} = 0$ for $\mathsf{V} = \mathsf{S}_k$ or $\mathsf{S}_{j,k}$. Instead, only the Tate twists $\mathsf{L}^i\mathsf{S}_k$ and $\mathsf{L}^i \mathsf{S}_{j,k}$ with $i > 0$ can appear, see \cite[Section 7]{BF}. This means fewer relations from point counts are needed to solve the linear system and determine the $a_{\mathsf{V}}$.
\end{rem}

Although these results rely on Conjecture \ref{lc}, they are known unconditionally in some cases. For example, in genus $2$, Conjecture \ref{lc} (and an appropriate generalization to all weights) follows from work of Petersen \cite{Pg2}. Moreover, by Theorem \ref{lowwt}, Conjecture \ref{lc} holds whenever $\dim \Mb_{g,n} \leq 15$.  In particular, the results for $\#\Mb_{3,n}$ are unconditional for $n \leq 9$. In fact this can be improved further to $n \leq 11$, as we now explain. Using the open circles in the chart in Section \ref{sec:taut} and Lemma \ref{open-ckgp}, all classes in $W_k H^k(\M_{3,n})$ are tautological, hence of Tate type, for $n \leq 11$. Then Petersen's results in genus $2$ ensures that the image of push forward from the boundary is also among the allowed Galois representations, see \cite[Theorem 1.9]{ste} for more details.

\medskip
As a final illustration of how the different invariants and results we have considered relate to each other, we give a proof of Proposition \ref{samir8}.
\begin{proof}[Proof of Proposition \ref{samir8}]
We first prove the claim that $H^*(\Mb_{3,9}) = RH^*(\Mb_{3,9})$.
In \cite[Theorem 8]{C}, it is shown that all even cohomology of $\Mb_{3,9}$ is tautological. However, the results of \cite{BF}, show that $\#\Mb_{3,9}(\ff_q)$ is a polynomial in $q$. As discussed above, this result is unconditonal when $n = 9$.
 Hence by Theorem \ref{be}, all odd cohomology vanishes.

Now consider $\Mb_{4,7}$. We have an exact sequence
\[H^{k-2}(\widetilde{\partial \M_{4,7}}) \to H^{k}(\Mb_{4,7}) \to W_k H^k(\M_{4,7}).\]
 By the open circle in the chart in Section \ref{sec:taut} and Remark \ref{rem:ckgp}, combined with Lemma \ref{open-ckgp}, we see that $W_k H^k(\M_{4,7})$ is generated by tautological classes. Thus, it suffices to show that all classes pushed forward from the boundary are tautological. We have just shown that $\Mb_{3,9}$  has tautological cohomology ring. Now consider a separating boundary divisor which is the image of a gluing map
 $\Mb_{g_1,n_1+1} \times \Mb_{g_2,n_2+1} \to \Mb_{4, 7}$. Every possible $(g_i, n_i+1)$ occurs in a position where there is already a filled circle in the chart in Section \ref{sec:taut}.
 
 The proofs for $(g,n) = (5,5), (6,3), (7,1)$ all proceed similarly. An open circle ensures all classes in $W_k H^k(\M_{g,n})$ are tautological. The previous case of the proposition
ensures classes from the irreducible boundary divisor are tautological. All separating boundary divisors are the images of products of moduli spaces that are already known to have tautological cohomology rings.
\end{proof}

\end{document}